\documentclass[11pt]{article}

\RequirePackage{amsthm,amsmath}
\RequirePackage{natbib}
\RequirePackage[colorlinks,citecolor=blue,urlcolor=blue]{hyperref}
\usepackage{graphicx}
\usepackage[top=1.25in,bottom=1.25in,left=1in,right=1in]{geometry}
\usepackage{amsfonts,amscd,amssymb,bm,epsfig,epsf,bbm,constants,amssymb, prettyref, mathabx}

\numberwithin{equation}{section}
\theoremstyle{plain}
\newtheorem{theorem}{Theorem}[section]
\newtheorem{lemma}[theorem]{Lemma}
\newtheorem{corollary}[theorem]{Corollary}

\newtheorem{definition}[theorem]{Definition}

\newtheorem{remark}[theorem]{Remark}

\newcommand{\vct}[1]{\bm{#1}}
\newcommand{\mtx}[1]{\bm{#1}}

\newcommand{\lspan}{\operatorname{span}}

\newcommand{\diag}{\operatorname{diag}}

\newcommand{\px}{\mathbold{\Phi}}
\newcommand{\py}{\mathbold{\Psi}}
\newcommand{\pxk}{\mathbold{\Phi}_{1:k}}
\newcommand{\pyk}{\mathbold{\Psi}_{1:k}}

\newcommand{\normf}{\operatorname{F}}

\newcommand{\pxh}{\widehat{\bf \Phi}}
\newcommand{\pyh}{\widehat{\bf \Psi}}
\newcommand{\pxhk}{\widehat{\bf \Phi}_{1:k}}
\newcommand{\pyhk}{\widehat{\bf \Psi}_{1:k}}
\newcommand{\s}{\mtx{\Sigma}}
\newcommand{\sx}{\mtx{\Sigma}_x}
\newcommand{\sy}{\mtx{\Sigma}_y}
\newcommand{\sxy}{\mtx{\Sigma}_{xy}}

\newcommand{\sh}{\widehat{\mtx{\Sigma}}}
\newcommand{\sxh}{\widehat{\mtx{\Sigma}}_x}
\newcommand{\syh}{\widehat{\mtx{\Sigma}}_y}
\newcommand{\sxyh}{\widehat{\mtx{\Sigma}}_{xy}}
\newcommand{\syxh}{\widehat{\mtx{\Sigma}}_{yx}}

\newcommand{\wt}[1]{\widetilde{#1}}
\newcommand{\wh}[1]{\widehat{#1}}
\newcommand{\mc}[1]{\mathcal{#1}}

\newcommand{\norm}[1]{{\left\vert\kern-0.25ex\left\vert\kern-0.25ex\left\vert #1 
    \right\vert\kern-0.25ex\right\vert\kern-0.25ex\right\vert}}
\newcommand{\half}{1/2}

\newcommand{\lossmax}{\mathcal{L}_{\max}}
\newcommand{\lossave}{\mathcal{L}_{\operatorname{ave}}}
\newcommand{\lossgao}{\overline{\mathcal{L}}_{\operatorname{ave}}}
\newcommand{\losscai}{\overline{\mathcal{L}}_{\max}}
\newcommand{\lossjoint}{\overline{\mathcal{L}}_{\operatorname{joint}}}

\newcommand{\strangesection}[1]{\renewcommand{\thesection}{#1}\section}

\def\argmax{\mathop{\rm arg\, max}}

\newcommand{\bel}{\begin{eqnarray}\label}
\newcommand{\eel}{\end{eqnarray}}
\newcommand{\bes}{\begin{eqnarray*}}
\newcommand{\ees}{\end{eqnarray*}}
\newcommand{\bei}{\begin{itemize}}
\newcommand{\beiftnt}{\begin{itemize}\footnotesize}
\newcommand{\eei}{\end{itemize}}

\def\benu{\begin{enumerate}}
\def\eenu{\end{enumerate}}

\def\argmax{\mathop{\rm arg\, max}}

\def\E{{\mathbb{E}}}
\def\P{{\mathbb{P}}}

\def\complex{\mathop{{\rm I}\kern-.58em\hbox{\rm C}}\nolimits}

\def\diag{\hbox{\rm diag}}

\def\trace{\hbox{\rm trace}}
\def\Cov{\hbox{\rm Cov}}

\def\Var{\hbox{\rm Var}}

\def\mathbold{\boldsymbol} 

\def\ba{\mathbold{a}}

\def\bA{\mathbold{A}}
\def\hbA{{\widehat{\bA}}}

\def\bB{\mathbold{B}}
\def\hbB{{\widehat{\bB}}}

\def\bD{\mathbold{D}}
\def\hbD{{\widehat{\bD}}}

\def\bE{\mathbold{E}}

\def\bI{\mathbold{I}}

\def\bJ{\mathbold{J}}

\def\bM{\mathbold{M}}

\def\bP{\mathbold{P}}

\def\bQ{\mathbold{Q}}

\def\bu{\mathbold{u}}

\def\bU{\mathbold{U}}\def\hbU{{\widehat{\bU}}}

\def\bv{\mathbold{v}}

\def\bV{\mathbold{V}}\def\hbV{{\widehat{\bV}}}

\def\bw{\mathbold{w}}

\def\bW{\mathbold{W}}

\def\bx{\mathbold{x}}

\def\by{\mathbold{y}}

\def\bZ{\mathbold{Z}}

\def\htheta{\widehat{\theta}}

\def\hlambda{\widehat{\lambda}}

\def\bLambda{\mathbold{\Lambda}}\def\hbLambda{{\widehat{\bLambda}}}

\def\bSigma{\mathbold{\Sigma}}

\def\bphi{\mathbold{\phi}}

\def\bpsi{\mathbold{\psi}}

\def\bOmega{\mathbold{\Omega}}

\def\0{\mathbold{0}}

\begin{document}

\title{Subspace Perspective on Canonical Correlation Analysis: Dimension Reduction and Minimax Rates} 
\author{Zhuang Ma~~~and~~~Xiaodong Li}

\date{}
\maketitle

\begin{abstract}
Canonical correlation analysis (CCA) is a fundamental statistical tool for exploring the correlation structure between two sets of random variables. In this paper, motivated by the recent success of applying CCA to learn low dimensional representations of high dimensional  objects, we propose two losses based on the principal angles between the model spaces spanned by the sample canonical variates and their population correspondents, respectively. We further characterize the non-asymptotic error bounds for the estimation risks under the proposed error metrics, which reveal how the performance of sample CCA depends adaptively on key quantities including the dimensions, the sample size, the condition number of the covariance matrices and particularly the population canonical correlation coefficients. The optimality of our uniform upper bounds is also justified by lower-bound analysis based on stringent and localized parameter spaces. To the best of our knowledge, for the first time our paper separates $p_1$ and $p_2$ for the first order term in the upper bounds without assuming the residual correlations are zeros. More significantly, our paper derives $(1- \lambda_k^2)(1-\lambda_{k+1}^2)/(\lambda_k - \lambda_{k+1})^2$ for the first time in the non-asymptotic CCA estimation convergence rates, which is essential to understand the behavior of CCA when the leading canonical correlation coefficients are close to $1$.
\end{abstract}



\strangesection{1}{Introduction}
Canonical correlation analysis (CCA), first introduced by \cite{hotelling36}, is a fundamental statistical tool to characterize the relationship between two groups of random variables and finds a wide range of applications across many different fields. For example, in genome-wide association study (GWAS), CCA is used to discover the genetic associations between the genotype data of single nucleotide polymorphisms (SNPs) and the phenotype data of gene expression levels \citep{witten2009penalized, chen2012structured}. In information retrieval, CCA is used to embed both the search space (e.g. images) and the query space (e.g. text) into a shared low dimensional latent space such that the similarity between the queries and  the candidates can be quantified \citep{rasiwasia2010new, gong2014multi}.  In natural language processing, CCA is applied to the word co-occurrence matrix and learns vector representations of the words which capture the semantics  \citep{dhillon11, faruqui2014improving}. Other applications, to name a few, include fMRI data analysis \citep{friman2003adaptive}, computer vision \citep{kim2007tensor} and speech recognition \citep{arora2013multi, wang2015deep}. 

The enormous empirical success motivates us to revisit the estimation problem of canonical correlation analysis. Two theoretical questions are naturally posed: What are proper error metrics to quantify the discrepancy between population CCA and its sample estimates? And under such metrics, what are the quantities that characterize the fundamental statistical limits? 

The justification of loss functions, in the context of CCA, has seldom appeared in the literature. From first principles that the proper metric to quantify the estimation loss should depend on the specific purpose of using CCA, we find that the applications discussed above mainly fall into two categories: identifying variables of interest and dimension reduction.

The first category, mostly in genomic research \citep{witten2009penalized, chen2012structured},  treats one group of variables as responses and the other group of variables as covariates. The goal is to discover the specific subset of the covariates that are most correlated with the responses. Such applications are featured by low signal-to-noise ratio and the interpretability of the results is the major concern.

In contrast, the second category is investigated extensively in statistical machine learning and engineering community where CCA is used to learn low dimensional latent representations of complex objects such as images \citep{rasiwasia2010new}, text \citep{dhillon11} and speeches \citep{arora2013multi}. These scenarios are usually accompanied with relatively high signal-to-noise ratio and the prediction accuracy, using the learned low dimensional embeddings as the new set of predictors, is of primary interest. In recent years, there has been a series of publications establishing fundamental theoretical guarantees for CCA to achieve sufficient dimension reduction (\cite{kakade07, foster08, SridharanK08, fukumizu2009kernel, chaudhuri2009multi} and many others).

In this paper, we aim to address the problems raised above by treating CCA as a tool for dimension reduction.

\subsection{Population and Sample CCA}
Suppose $\bx = [X_1, \ldots, X_{p_1}]^\top \in\mathbb{R}^{p_1}$ and $\by = [Y_1, \ldots, Y_{p_2}]^\top \in\mathbb{R}^{p_2}$ are two sets of variates with the joint covariance matrix
\begin{equation}
\label{eq:covariance}
\text{Cov}\left(\begin{bmatrix} \vct{x} \\ \vct{y} \end{bmatrix}\right) = \mtx{\Sigma} := \begin{bmatrix} \mtx{\Sigma}_{x} & \mtx{\Sigma}_{xy} \\ \mtx{\Sigma}_{xy}^\top & \mtx{\Sigma}_y \end{bmatrix}. 
\end{equation}
For simplicity, we assume 
\[
\E(X_i) =0,~i =1, \ldots, p_1, \quad \E(Y_j) = 0,~j = 1, \ldots, p_2.
\]
On the population level, CCA is designed to extract the most correlated linear combinations between two sets of random variables sequentially: The $i$th pair of \textbf{canonical variables} $U_i = \bphi_i^\top\bx$ and $V_i = \bpsi_i^\top\by$ maximizes
\[
\lambda_i = \text{Corr}(U_i, V_i)
\]
such that $U_i$ and $V_i$ have unit variances and they are uncorrelated to all previous pairs of canonical variables. Here $(\bphi_i, \bpsi_i)$ is called the $i$th pair of \textbf{canonical loadings} and $\lambda_i$ is the $i$th \textbf{canonical correlation}.\\

It is well known in multivariate statistical analysis that the canonical loadings can be found recursively by the following criterion:
\begin{equation}
\begin{aligned}
(\bphi_i, \bpsi_i)= &\argmax  &&   \bphi^\top\sxy\bpsi
\\
~&\text{subject to} &&\bphi^\top\sx\bphi=1,~ \bpsi^\top\sy\bpsi=1;
\\
~&~&& \bphi^\top\sx \bphi_j=0, ~\bpsi^\top\sy\bpsi_j=0,~ \forall~1\leq j\leq i-1.
\end{aligned}
\label{eq-def-cca}
\end{equation}
Although this criterion is a nonconvex optimization, it can be obtained easily by spectral methods: Define $\px:=[\bphi_1, \cdots, \bphi_{p_1 \wedge p_2}]$, $\py:=[\bpsi_1, \cdots, \bpsi_{p_1 \wedge p_2}]$ and $\bLambda:=\diag(\lambda_1, \cdots, \lambda_{p_1 \wedge p_2})$. Then $\lambda_1, \ldots, \lambda_{p_1 \wedge p_2}$ are singular values of $\sx^{-\half}\sxy\sy^{-\half}$, and $\sx^{\half}\px, \sy^{\half}\py$ are actually left and right singular vectors of $\sx^{-\half}\sxy\sy^{-\half}$, respectively.\\

%
%
%


\subsection{Canonical variables versus canonical loadings}

%
%

For any given estimates of the leading $k$ canonical loadings, denoted by $\{(\widehat{\vct{\phi}}_i, \widehat{\vct{\psi}}_i)\}_{i=1}^k$, the corresponding estimates for the canonical variables can be represented by 
\[
\widehat{U}_i = \widehat{\vct{\phi}}_i^\top \vct{x}, \quad \widehat{V}_i = \widehat{\vct{\psi}}_i^\top \vct{y}, \quad i = 1 \ldots, p_1 \wedge p_2.
\]
To quantify the estimation loss, generally speaking, we can either focus on measuring the difference between the canonical loadings $\{(\vct{\phi}_i, \vct{\psi}_i)\}_{i=1}^k$ and $\{(\widehat{\vct{\phi}}_i, \widehat{\vct{\psi}}_i)\}_{i=1}^k$ or measuring the difference between the canonical variables $\{(U_i, V_i)\}_{i=1}^k$ and $\{(\widehat{U}_i, \widehat{V}_i)\}_{i=1}^k$. Here $\vct{x}, \vct{y}$ in the definition of $\{(U_i, V_i)\}_{i=1}^k$ and $\{(\widehat{U}_i, \widehat{V}_i)\}_{i=1}^k$ are independent of the samples based on which $\{(\widehat{\vct{\phi}}_i, \widehat{\vct{\psi}}_i)\}_{i=1}^k$ are constructed. Therefore, for the discrepancy between the canonical variables, there is an extra layer of randomness.


As discussed above, in modern machine learning applications such as natural language processing and information retrieval, the leading sample canonical loadings are used for dimension reduction, i.e., for a new observation $(\vct{x}_0, \vct{y}_0)$, ideally we hope to use the corresponding values of the canonical variables $(u_i = \vct{\phi}_i^\top \vct{x}_0)_{i=1}^k$ and $(v_i = \vct{\psi}_i^\top \vct{y}_0)_{i=1}^k$ to represent the observation in a low dimension space. Empirically, the actual low dimensional representations are $(\hat{u}_i = \widehat{\vct{\phi}}_i^\top \vct{x}_0)_{i=1}^k$ and $(\hat{v}_i = \widehat{\vct{\psi}}_i^\top \vct{y}_0)_{i=1}^k$. Therefore, the discrepancy between the ideal dimension reduction and actual dimension reduction should be explained by how well $\{(\widehat{U}_i, \widehat{V}_i)\}_{i=1}^k$ approximate $\{(U_i, V_i)\}_{i=1}^k$. Consequently, we choose to quantify the difference between the sample and population canonical variables instead of the canonical loadings.

\subsection{Linear span}
However, there are still many options to quantify how well the sample canonical variables approximate their population correspondents. To choose suitable losses, it is convenient to come back to specific applications to get some inspiration. \\

Motivated by applications in natural language processing and information retrieval, the model of multi-view sufficient dimension reduction has been studied in \cite{foster08}. Roughly speaking, a statistical model was proposed by \cite{foster08} to study how to predict $Z$ using two sets of predictors denoted by $\vct{x} = [X_1, \ldots, X_{p_1}]^\top$ and $\vct{y} = [Y_1, \ldots, Y_{p_2}]^\top$, where the joint covariance of $(Z, \vct{x}, \vct{y})$ is
\[
\text{Cov}\left(\begin{bmatrix} \vct{x} \\ \vct{y} \\ Z \end{bmatrix}\right) = \begin{bmatrix} \mtx{\Sigma}_x & \mtx{\Sigma}_{xy} & \vct{\sigma}_{xz} \\ \mtx{\Sigma}_{xy}^\top & \mtx{\Sigma}_y & \vct{\sigma}_{yz} \\ \vct{\sigma}_{xz}^\top & \vct{\sigma}_{yz}^\top & \sigma_z^2 \end{bmatrix}. 
\]
It was proven in \cite{foster08} that under certain assumptions, the leading $k$ canonical variables $U_1, \ldots U_k$ are sufficient dimension reduction for the linear prediction of $Z$; That is, the best linear predictor of $Z$ based on $X_1, \ldots, X_{p_1}$ is the same as the best linear predictor based on $U_1, \ldots U_k$. (Similarly, the best linear predictor of $Z$ based on $Y_1, \ldots, Y_{p_2}$ is the same as the best linear predictor based on $V_1, \ldots V_k$.) 

Notice that the best linear predictor is actually determined by the set of all linear combinations of $U_1, \ldots, U_k$ (referred to as the ``model space" in the literature of linear regression for prediction), which we denote as $\lspan(U_1, \ldots, U_k)$. Inspired by \cite{foster08}, we propose to quantify the discrepancy between $\{U_i\}_{i=1}^k$ and $\{\widehat{U}_i\}_{i=1}^k$ by the discrepancy between the corresponding subspaces $\lspan(\widehat{U}_1, \ldots, \widehat{U}_k)$ and $\lspan(U_1, \ldots, U_k)$ (and similarly measure the difference between  $\{V_i\}_{i=1}^k$ and $\{\widehat{V}_i\}_{i=1}^k$ by the distance between $\lspan(\widehat{V}_1, \ldots, \widehat{V}_k)$ and $\lspan(V_1, \ldots, V_k)$).


\subsection{Hilbert spaces and principal angles}
In this section, we define the discrepancy between $\widehat{\mathcal{M}}_{(U, k)} = \lspan(\widehat{U}_1, \ldots, \widehat{U}_k)$ and $\mathcal{M}_{(U, k)} = \lspan(U_1, \ldots, U_k)$ by introducing a Hilbert space. Noting that for any given sample $\{(x_i, y_i)\}_{i=1}^n$, both $\widehat{\mathcal{M}}_{(U, k)}$ and $\mathcal{M}_{(U, k)}$ are composed by linear combinations of $X_1, \ldots, X_{p_1}$. Denote the set of all possible linear combinations as
\begin{equation}
\label{eq:hilbert}
\mathcal{H}=\lspan(X_1, \ldots, X_{p_1}).
\end{equation}
Moreover, for any $X_1, X_2\in \mathcal{H}$, we define a bilinear function $\langle X_1, X_2 \rangle$ $:=\text{Cov}(X_1, X_2)$ $=\E(X_1X_2)$. It is easy to show that $\langle \cdot, \cdot \rangle$ is an inner product and $(\mathcal{H}, \langle \cdot, \cdot \rangle)$ is a $p_1$-dimensional Hilbert space, which is isomorphic to $\mathbb{R}^{p_1}$.\\

With the natural covariance-based inner product, we know both $\widehat{\mathcal{M}}_{(U, k)}$ and $\mathcal{M}_{(U, k)}$ are subspaces of $\mathcal{H}$, so it is natural to define their discrepancy based on their principal angles $\frac{\pi}{2} \geq \theta_1 \geq \ldots \geq \theta_k \geq 0$. In the literature of statistics and linear algebra, two loss functions are usually used
\[
\lossmax(\lspan(\widehat{U}_1, \ldots, \widehat{U}_k),~ \lspan(U_1, \ldots, U_k)) = \sin^2 (\theta_1)
\]
and
\[
\lossave(\lspan(\widehat{U}_1, \ldots, \widehat{U}_k),~ \lspan(U_1, \ldots, U_k)) = \frac{1}{k}(\sin^2 (\theta_1)+ \ldots + \sin^2(\theta_k))
\]
In spite of a somewhat abstract definition, we have the following clean formula for these two losses:
\begin{theorem}
\label{thm:formula}
Suppose for any $p_1 \times k$ matrix $\mtx{A}$, $\mtx{P}_{\mtx{A}}$ represents the orthogonal projector onto the column span of $\mtx{A}$. Assume the observed sample is fixed. Then
\begin{align}
\label{eq:formula1}
\lossave(\lspan(\widehat{U}_1, \ldots, \widehat{U}_k), \lspan(U_1, \ldots, U_k)) &= \frac{1}{2k}\left\|\mtx{P}_{\mtx{\Sigma}_x^{\half}\pxhk} - \mtx{P}_{\mtx{\Sigma}_x^{\half}\pxk}\right\|_F^2   \nonumber
\\
&= \frac{1}{k}\left\|\left(\mtx{I}_{p_1} - \mtx{P}_{\mtx{\Sigma}_x^{\half}\pxk}\right)\mtx{P}_{\mtx{\Sigma}_x^{\half}\pxhk}\right\|_F^2
\\
&= \frac{1}{k} \min_{\mtx{Q} \in \mathbb{R}^{k \times k}} \E \left[\| \vct{u}^\top - \widehat{\vct{u}}^\top \mtx{Q}\|_2^2\right] \nonumber
\\
&:= \lossave(\mtx{\Phi}_{1:k}, \widehat{\mtx{\Phi}}_{1:k}) \nonumber
\end{align}
and
\begin{align}
\label{eq:formula2}
\lossmax(\lspan(\widehat{U}_1, \ldots, \widehat{U}_k), \lspan(U_1, \ldots, U_k)) & = \left\|\mtx{P}_{\mtx{\Sigma}_x^{\half}\pxhk} - \mtx{P}_{\mtx{\Sigma}_x^{\half}\pxk}\right\|^2 \nonumber
\\
& = \left\|\left(\mtx{I}_{p_1} - \mtx{P}_{\mtx{\Sigma}_x^{\half}\pxk}\right)\mtx{P}_{\mtx{\Sigma}_x^{\half}\pxhk}\right\|^2\\
& = \max_{\vct{g}\in\mathbb{R}^k}\min_{\mtx{Q} \in \mathbb{R}^{k \times k}} \E \left[ \left(\left(\vct{u}^\top - \widehat{\vct{u}}^\top \mtx{Q}\right)\vct{g}\right)^2\right] \nonumber
\\
& := \lossmax(\mtx{\Phi}_{1:k}, \widehat{\mtx{\Phi}}_{1:k}). \nonumber
\end{align}
%

	Here $\pxk = [\vct{\phi}_1, \ldots, \vct{\phi}_k]$ is a $p_1 \times k$ matrix consisting of the leading $k$ population canonical loadings for $\vct{x}$, and $\pxhk$ is its estimate based on a given sample. Moreover $\vct{u}^\top:=(U_1, \ldots, U_k)$ and $\hat{\vct{u}}^\top:=(\widehat{U}_1, \ldots, \widehat{U}_k)$.

\end{theorem}


\subsection{Uniform upper bounds and minimax rates}
The most important contribution of this paper is to establish sharp upper bounds for the estimation/prediction  of CCA based on the  proposed subspace losses $\lossmax(\mtx{\Phi}_{1:k}, \widehat{\mtx{\Phi}}_{1:k})$ and $\lossave(\mtx{\Phi}_{1:k}, \widehat{\mtx{\Phi}}_{1:k})$.
It is noteworthy that both upper bounds hold uniformly for all invertible $\mtx{\Sigma}_x, \mtx{\Sigma}_y$ provided $n > C(p_1 +p_2)$ for some numerical constant $C$. Furthermore, in order to justify the sharpness of these bounds, we also establish minimax lower bounds under a family of stringent and localized parameter spaces. These results will be detailed in \prettyref{sec:theory}.

\subsection{Notations and the Organization}
Throughout the paper, we use lower-case and upper-case non-bolded letters to represent fixed and random variables, respectively. We also use lower-case and upper-case bold letters to represent vectors (which could be either deterministic or random) and matrices, respectively. For any matrix $\bU\in\mathbb{R}^{n\times p}$ and vector $\bu\in\mathbb{R}^p$, $\|\bU\|, \|\bU\|_F$ denotes operator (spectral) norm and Frobenius norm respectively, $\|\bu\|$ denotes the vector $l_2$ norm, $\bU_{1:k}$ denotes the submatrix consisting of the first $k$ columns of $\bU$, and $\bP_{\bU}$ stands for the projection matrix onto the column space of $\bU$. Moreover, we use $\sigma_{\max}(\bU)$ and $\sigma_{\min}(\bU)$ to represent the largest and smallest singular value of $\bU$ respectively, and $\kappa(\bU)=\sigma_{\max}(\bU)/\sigma_{\min}(\bU)$ to denote the condition number of the matrix. We use $\bI_{p}$ for the identity matrix of dimension $p$ and $\bI_{p, k}$ for the submatrix composed of the first $k$ columns of $\bI_p$. Further, $\mathcal{O}(m, n)$ (and simply $\mathcal{O}(n)$ when $m=n$) stands for the set of $m\times n$ matrices with orthonormal columns and $\mathbb{S}_+^p$ denotes  the set of $p\times p$ strictly positive definite matrices. For a random vector $\bx\in\mathbb{R}^p$, $\lspan(\bx^\top)=\{\bx^\top\bw, \bw\in\mathbb{R}^p\}$ denotes the subspace of all the linear combinations of $\bx$. Other notations will be specified within the corresponding context. 

In the following, we will introduce our main upper and lower bound results in \prettyref{sec:theory}. To highlight our contributions in the new loss functions and theoretical results, we will compare our results to existing work in the literature in \prettyref{sec:contributions}. All proofs are deferred to 
\prettyref{sec:proofs}.

\strangesection{2}{Theory}
\label{section-theory}
\label{sec:theory}
In this section, we introduce our main results on non-asymptotic upper and lower bounds for estimating CCA under the proposed loss functions. It is worth recalling that $\lambda_1, \ldots, \lambda_{p_1 \wedge p_2}$ are singular values of $\sx^{-\half}\sxy\sy^{-\half}$.

It is natural to estimate population CCA by its sample counterparts. Similar to equation~\eqref{eq-def-cca}, the \textbf{sample canonical loadings} are defined recursively by 
\begin{equation}
\begin{aligned}
(\wh{\bphi}_i, \wh{\bpsi}_i)= &\argmax  &&   \bphi^\top \widehat{\mtx{\Sigma}}_{xy} \bpsi
\\
~&\text{subject to} &&\bphi^\top \widehat{\mtx{\Sigma}}_x \bphi=1, ~\bpsi^\top \widehat{\mtx{\Sigma}}_y \bpsi=1;
\\
~&~&& \bphi^\top\widehat{\mtx{\Sigma}}_x  \bphi_j=0, ~\bpsi^\top \widehat{\mtx{\Sigma}}_y \bpsi_j=0,~ \forall~1\leq j\leq i-1.
\end{aligned}
\label{eq:sample_cca}
\end{equation}
where $\sxh, \syh, \sxyh$ are the sample covariance matrices. The \textbf{sample canonical variables} are defined as the following linear combinations by the sample canonical loadings:
\[
\widehat{U}_i = \widehat{\vct{\phi}}_i^\top \vct{x}, \quad \widehat{V}_i = \widehat{\vct{\psi}}_i^\top \vct{y}, \quad i = 1 \ldots, p_1 \wedge p_2.
\]
We prove the following upper bound for the estimate based on sample CCA.

\begin{theorem}(Upper bound)
\label{thm:upper}
Suppose $\begin{bmatrix} \vct{x}\\ \vct{y} \end{bmatrix} \sim \mathcal{N}(\vct{0}, \mtx{\Sigma})$ where $\mtx{\Sigma}$ is defined as in \prettyref{eq:covariance}. Assume $\mtx{\Sigma}_x$ and $\mtx{\Sigma}_y$ are invertible. Moreover, assume $\lambda_k > \lambda_{k+1}$ for some predetermined $k$. Then there exist universal positive constants $\gamma, C, C_0$ such that if $n\geq C(p_1+p_2)$, the top-$k$ sample canonical coefficients matrix $\pxhk$ satisfies
\begin{gather*}
\E\left[ \lossmax(\mtx{\Phi}_{1:k}, \widehat{\mtx{\Phi}}_{1:k})\right] \leq C_0 \left[\frac{(1-\lambda_k^2)(1-\lambda_{k+1}^2)}{(\lambda_k - \lambda_{k+1})^2}\frac{p_1}{n} + \frac{(p_1+p_2)^2}{n^2(\lambda_k - \lambda_{k+1})^4} + e^{-\gamma (p_1 \wedge p_2)}\right]\\
\E\left[ \lossave(\mtx{\Phi}_{1:k}, \widehat{\mtx{\Phi}}_{1:k})\right] \leq C_0 \left[\frac{(1-\lambda_k^2)(1-\lambda_{k+1}^2)}{(\lambda_k - \lambda_{k+1})^2}\frac{p_1-k}{n} + \frac{(p_1+p_2)^2}{n^2(\lambda_k - \lambda_{k+1})^4} + e^{-\gamma (p_1 \wedge p_2)}\right]
\end{gather*}
The upper bounds for $\pyhk$ can be obtained by switching $p_1$ and $p_2$. 
\end{theorem}

Since we pursue a nonasymptotic theoretical framework for CCA estimates, and the loss functions we propose are nonstandard in the literature, the standard minimax lower bound results in parametric maximum likelihood estimates do not apply straightforwardly. Instead, we turn to the nonparametric minimax lower bound frameworks, particularly those in PCA and CCA; See, e.g., \cite{vu2013minimax, cai2013sparse, gao2015minimax}. Compared to these existing works, the technical novelties of our results and proofs are summarized in Sections \ref{sec:lower_bound} and \ref{sec:lower-proof}.


We define the parameter space $\mathcal{F}(p_1, p_2, k, \lambda_{k}, \lambda_{k+1}, \kappa_1, \kappa_2)$ as the collection of joint covariance matrices $\s$ satisfying 
%
\begin{enumerate}
\item $\kappa(\sx)=\kappa_1$ and $\kappa(\sy)=\kappa_2$;
\item $0\leq \lambda_{p_1\wedge p_2}\leq \cdots\leq \lambda_{k+1}<\lambda_k\leq \cdots\leq \lambda_1\leq 1$.
\end{enumerate}
We deliberately set $\kappa(\sx)=\kappa_1, \kappa(\sy)=\kappa_2$ to demonstrate that the lower bound is independent of the condition number. For the rest of the paper, we will use the shorthand $\mathcal{F}$ to represent this parameter space for simplicity. 
\begin{theorem}(Lower bound)
\label{thm:lower}
There exists a universal constant $c$ independent of $n, p_1, p_2$ and $\s$ such that
\begin{gather*}
\inf\limits_{\pxhk}\sup_{\s\in  \mathcal{F}}  \E\left[ \lossmax(\mtx{\Phi}_{1:k}, \widehat{\mtx{\Phi}}_{1:k})\right] \geq c^2\left\{\left(\frac{(1-\lambda_{k}^2)(1-\lambda_{k+1}^2)}{(\lambda_k - \lambda_{k+1})^2}\frac{p_1-k}{n}\right)\wedge 1\wedge \frac{p_1-k}{k}\right\}
\\
\inf\limits_{\pxhk}\sup_{\s\in  \mathcal{F}}  \E\left[ \lossave(\mtx{\Phi}_{1:k}, \widehat{\mtx{\Phi}}_{1:k})\right]\geq c^2\left\{\left(\frac{(1-\lambda_{k}^2)(1-\lambda_{k+1}^2)}{(\lambda_k - \lambda_{k+1})^2}\frac{p_1-k}{n}\right)\wedge 1\wedge \frac{p_1-k}{k}\right\}.
\end{gather*}
The lower bounds for $\pyhk$ can be obtained by replacing $p_1$ with $p_2$. 
\end{theorem}




\begin{corollary}
\label{thm::samplesize}
When $p_1, p_2\geq (2k) \vee C(\log n)$ and
\begin{equation}
n \geq C\frac{(p_1+p_2)(1+p_2/p_1)}{(\lambda_k - \lambda_{k+1})^2(1-\lambda_{k}^2)(1-\lambda_{k+1}^2)}
\label{eq:sample-size}
\end{equation}
for some universal positive constant $c$, the minimax rates can be characterized by 
\begin{gather*}
\inf\limits_{\pxhk}\sup_{\s\in  \mathcal{F}}  \E\left[ \lossmax(\mtx{\Phi}_{1:k}, \widehat{\mtx{\Phi}}_{1:k}) \right] \asymp \frac{(1-\lambda_{k}^2)(1-\lambda_{k+1}^2)}{(\lambda_k - \lambda_{k+1})^2}\frac{p_1}{n},\\
\inf\limits_{\pxhk}\sup_{\s\in  \mathcal{F}}  \E\left[ \lossave(\mtx{\Phi}_{1:k}, \widehat{\mtx{\Phi}}_{1:k}) \right] \asymp \frac{(1-\lambda_{k}^2)(1-\lambda_{k+1}^2)}{(\lambda_k - \lambda_{k+1})^2}\frac{p_1}{n}.
\end{gather*}
\end{corollary}

\strangesection{3}{Related Work and Our Contributions}
\label{sec:contributions}
\label{sec:contributions}
Recently, the non-asymptotic rate of convergence of CCA has been studied by \cite{gao2015minimax, gao2014sparse} under a sparse setup and by \cite{cai2016rate} under the usual non-sparse setup. \cite{cai2016rate} appeared on arXiv almost at the same time as the first version of our paper was posted. In this section, we state our contributions by detailed comparison with these works. 
\subsection{Novel loss funcitons}
\label{sec:loss_comparison}
We proposed new loss functions based on the principal angles between the subspace spanned by the population canonical variates and the subspace spanned by the estimated canonical variates. 
In contrast, \cite{gao2014sparse} proposed and studied the loss $\lossgao$; \cite{cai2016rate} proposed $\losscai$ and studied both $\lossgao$ and $\losscai$, where 
\begin{equation*}
\begin{aligned}
 \lossgao(\mtx{\Phi}_{1:k}, \widehat{\mtx{\Phi}}_{1:k}) & = \min_{\mtx{Q} \in \mathcal{O}(k, k)} \E \left[\| \vct{x}^\top \mtx{\Phi}_{1:k} - \vct{x}^\top \widehat{\mtx{\Phi}}_{1:k} \mtx{Q}\|_2^2~\Big\vert~\widehat{\mtx{\Phi}}_{1:k}\right], \\
 \losscai(\mtx{\Phi}_{1:k}, \widehat{\mtx{\Phi}}_{1:k}) & = \max_{\vct{g}\in\mathbb{R}^k, |\vct{g}| = 1}\min_{\mtx{Q} \in \mathcal{O}(k, k)}\E \left[  \left( \left(\vct{x}^\top \mtx{\Phi}_{1:k} - \vct{x}^\top \widehat{\mtx{\Phi}}_{1:k} \mtx{Q}\right)\vct{g} \right)^2~\Big\vert~\widehat{\mtx{\Phi}}_{1:k}\right].
\end{aligned}
\end{equation*}
$\lossgao$ and $\losscai$ resemble our loss functions $\lossave$ and $\lossmax$ respectively. By \prettyref{thm:formula}, we also have
\begin{equation*}
\begin{aligned}
 \lossave(\mtx{\Phi}_{1:k}, \widehat{\mtx{\Phi}}_{1:k}) &= 2\min_{\mtx{Q} \in \mathbb{R}^{k \times k}} \E \left[\| \vct{x}^\top \mtx{\Phi}_{1:k} - \vct{x}^\top \widehat{\mtx{\Phi}}_{1:k} \mtx{Q}\|_2^2 ~\Big\vert~ \widehat{\mtx{\Phi}}_{1:k}\right] \\
 \lossmax(\mtx{\Phi}_{1:k}, \widehat{\mtx{\Phi}}_{1:k}) &= \max_{\vct{g}\in\mathbb{R}^k, |\vct{g}| = 1}\min_{\mtx{Q} \in \mathbb{R}^{k \times k}} \E \left[  \left( \left(\vct{x}^\top \mtx{\Phi}_{1:k} - \vct{x}^\top \widehat{\mtx{\Phi}}_{1:k} \mtx{Q}\right)\vct{g} \right)^2~\Big\vert~\widehat{\mtx{\Phi}}_{1:k}\right]
\end{aligned}
\end{equation*}
By these two expressions, we can easily obtain
\begin{equation}
\label{eq:connection}
\begin{aligned}
	\lossave(\mtx{\Phi}_{1:k}, \widehat{\mtx{\Phi}}_{1:k}) &\leq 2\lossgao(\mtx{\Phi}_{1:k}, \widehat{\mtx{\Phi}}_{1:k}) \\
\lossmax(\mtx{\Phi}_{1:k}, \widehat{\mtx{\Phi}}_{1:k}) & \leq \losscai(\mtx{\Phi}_{1:k}, \widehat{\mtx{\Phi}}_{1:k})
\end{aligned}
\end{equation}
However, $\lossave(\mtx{\Phi}_{1:k}, \widehat{\mtx{\Phi}}_{1:k})$ and $\lossgao(\mtx{\Phi}_{1:k}, \widehat{\mtx{\Phi}}_{1:k})$ are not equivalent up to a constant. Neither are $\lossmax(\mtx{\Phi}_{1:k}, \widehat{\mtx{\Phi}}_{1:k})$ and $\losscai(\mtx{\Phi}_{1:k}, \widehat{\mtx{\Phi}}_{1:k})$. In fact, we can prove that as long as $n > \max(p_1, p_2)$, if $\lambda_k=1>\lambda_{k+1}$, then
\[
\lossave(\mtx{\Phi}_{1:k}, \widehat{\mtx{\Phi}}_{1:k}) = \lossmax(\mtx{\Phi}_{1:k}, \widehat{\mtx{\Phi}}_{1:k}) = 0,
\]
while almost surely $\lossgao(\mtx{\Phi}_{1:k}, \widehat{\mtx{\Phi}}_{1:k}) \neq 0$ and $\losscai(\mtx{\Phi}_{1:k}, \widehat{\mtx{\Phi}}_{1:k}) \neq 0$.

To illustrate this comparison, we can consider the following very simple simulation: Suppose $p_1 = p_2 = 2$, $n=3$ and $\mtx{\Sigma}_x = \begin{bmatrix} 1 & 0\\ 0 & 1\end{bmatrix}$ and $\mtx{\Sigma}_y = \begin{bmatrix} 1 & 0 \\ 0 & 1\end{bmatrix}$ and $\mtx{\Sigma}_{xy} = \begin{bmatrix} 1 & 0\\ 0 & 0.5\end{bmatrix}$. In this setup, we know the population canonical correlation coefficients are $\lambda_1 = 1$ and $\lambda_2 = 0.5$, and the leading canonical loadings are $\vct{\phi}_1 = \begin{bmatrix} 1 \\ 0 \end{bmatrix}$ and $\vct{\psi}_1 = \begin{bmatrix} 1 \\ 0 \end{bmatrix}$. In our simulation, we generated the following data matrices
\[
\mtx{X}= \begin{bmatrix} 0.0736 & 1.5496 \\ 1.5390 & -0.0415 \\ 0.9331 & -0.4776\end{bmatrix}
\]
and
\[
\mtx{Y}= \begin{bmatrix} 0.0736 & 2.8982 \\ 1.5390 & -1.2214 \\ 0.9331 & 2.5931\end{bmatrix}.
\]
Furthermore, we can obtain the sample canonical correlations $\widehat{\lambda}_1 = 1$ and $\widehat{\lambda}_2 = 0.5210$, as well as the leading sample canonical loadings $\widehat{\vct{\phi}}_1 = \begin{bmatrix} -0.9616 \\ 0 \end{bmatrix}$ and $\widehat{\vct{\psi}}_1 = \begin{bmatrix} -0.9616 \\ 0 \end{bmatrix}$. Then $\lossave(\vct{\phi}_1, \widehat{\vct{\phi}}_1) = \lossmax(\vct{\phi}_1, \widehat{\vct{\phi}}_1) =  0$ while $\lossgao(\vct{\phi}_1, \widehat{\vct{\phi}}_1)\neq 0, \losscai(\vct{\phi}_1, \widehat{\vct{\phi}}_1) \neq 0$. 

This numerical example clearly shows that the sample CCA can exactly identify that among all linear combinations of $X_1$ and $X_2$ and all linear combinations of $Y_1$ and $Y_2$, $a X_1$ and $b Y_1$ are mostly correlated. Our loss functions $\lossave$ and $\lossmax$ do characterize this exact identification, whereas $\lossgao$ and $\losscai$ do not.

Moreover, the following joint loss was studied in \cite{gao2015minimax}:
\[
\lossjoint\left(\left(\mtx{\Phi}_{1:k}, \mtx{\Psi}_{1:k}\right), \left(\widehat{\mtx{\Phi}}_{1:k}, \widehat{\mtx{\Psi}}_{1:k}\right)\right) = \E \left[ \left\| \pxhk\pyhk^\top - \pxk\pyk^\top\right\|_{\normf}^2 \right]. 
\]
Similarly, $\lossjoint\left(\left(\mtx{\Phi}_{1:k}, \mtx{\Psi}_{1:k}\right), \left(\widehat{\mtx{\Phi}}_{1:k}, \widehat{\mtx{\Psi}}_{1:k}\right)\right)\neq 0$ almost surely under the special case $\lambda_k = 1 > \lambda_{k+1}$. 
\subsection{Sharper upper bounds}
\label{sec:upper_bound}
Regardless of loss functions, we explain in the following why \prettyref{thm:upper} implies sharper upper bounds than the existing rates in \cite{gao2015minimax}, \cite{gao2014sparse} and \cite{cai2016rate} under the nonsparse case. Our discussion is focused on $\lossave$ in the following discussion while the discussion for $\lossmax$ is similar.

Notice that if we only apply Wedin's sin-theta law, i.e., replacing the fine bound \prettyref{lem:operator_bound} with the rough bound \prettyref{lem:estimation} (also see \cite{gao2015minimax} for similar ideas), we can obtain the following rough bound:
\begin{equation}
\label{eq:rough_bound}
\E\left[ \lossave(\mtx{\Phi}_{1:k}, \widehat{\mtx{\Phi}}_{1:k})\right] \leq C_0 \left[\frac{p_1 + p_2}{n(\lambda_k - \lambda_{k+1})^2}\right].
\end{equation}

In order to decouple the estimation error bound of $\widehat{\mtx{\Phi}}_{1:k}$ from $p_2$, both \cite{gao2014sparse} and \cite{cai2016rate} assume the residual canonical correlations are zero, i.e., 
\[
\lambda_{k+1} = \ldots = \lambda_{p_1 \wedge p_2} =0.
\]
This assumption is essential for proofs in both \cite{gao2014sparse} and \cite{cai2016rate} under certain sample size conditions. 
We got rid of this assumption by developing new proof techniques and these techniques actually work for $\lossgao, \losscai$ as well. 
A detailed comparison between our result and that in \cite{cai2016rate} is summarized in Table~\prettyref{table: comparison} (The results of \cite{gao2014sparse} in the non-sparse regime can be implied by \cite{cai2016rate} under milder sample size conditions).

\begin{table}[htb]
\begin{center}
\begin{tabular}{|c|c|c|c|}

\hline
~  & Cai and Zhang 2016 &  Our work\\
\hline
 Loss function	&  $\lossgao (\geq \lossave)$ & $\lossave$  \\  
 \hline
 Sample size	&  $n>C \left(\frac{p_1+\sqrt{p_1p_2} }{\lambda_k^2} + \frac{p_2}{\lambda_k^{4/3}}\right)$ & $n>C(p_1 + p_2)$  \\  
 \hline
 $\lambda_{k+1} = \cdots = \lambda_{p_1} = 0$ 	&  Yes & No  
 \\ \hline
Upper Bound Rates	&  $\frac{p_1}{n\lambda_k^2} + \frac{p_1p_2}{n^2\lambda_k^4}$ & $\frac{(1-\lambda_k^2)(1 - \lambda_{k+1}^2)}{(\lambda_k - \lambda_{k+1})^2}\frac{p_1 - k}{n} + \frac{(p_1 + p_2)^2}{n^2(\lambda_k - \lambda_{k+1})^4} + e^{-\gamma (p_1 \wedge p_2)}$ \\
\hline
\end{tabular}
\end{center}
\label{table: comparison}
\end{table}

Perhaps the most striking contribution of our upper bound is that we first derive the factors $(1-\lambda_k^2)$ and $(1-\lambda_{k+1}^2)$ in the literature of nonasymptotic CCA estimate. We now explain why these factors are essential when leading canonical correlation coefficients are close to $1$.

\subsubsection*{Example 1: $\lambda_{k}=1$ and $\lambda_{k+1}=0$} 
Consider the example that $k=1$, $p_1=p_2:=p \gg \log n$, $\lambda_{1} = 1$ and $\lambda_{2} = 0$. Then our bound rates $\frac{(1-\lambda_k^2)(1 - \lambda_{k+1}^2)}{(\lambda_k - \lambda_{k+1})^2}\frac{p_1 - k}{n} + \frac{(p_1 + p_2)^2}{n^2(\lambda_k - \lambda_{k+1})^4} + e^{-\gamma (p_1 \wedge p_2)}$
actually imply that
\[
\E \lossave(\vct{\phi}_1, \widehat{\vct{\phi}}_1)  \leq C\frac{p^2}{n^2},
\]
while the rates in \cite{gao2014sparse} and \cite{cai2016rate} imply that
\[
\E \lossave(\vct{\phi}_1, \widehat{\vct{\phi}}_1) \leq 2\E \lossgao(\vct{\phi}_1, \widehat{\vct{\phi}}_1) \leq C\frac{p}{n}.
\]
This shows that even under the condition $\lambda_{k+1} = 0$, under our loss $\lossave(\vct{\phi}_1, \widehat{\vct{\phi}}_1)$, our result could imply sharper convergence rates than that in \cite{gao2014sparse} and \cite{cai2016rate} if $\lambda_k =1$. 

Notice that as aforementioned, when $\lambda_k=1$, we can actually prove $\E \lossave(\vct{\phi}_1, \widehat{\vct{\phi}}_1) = 0$ through a separate argument. How to improve \prettyref{thm:upper} to imply this result is an open problem for future research.

\subsubsection*{Example 2: Both $\lambda_{k}$ and $\lambda_{k+1}$ are close to $1$} 
Consider the example that $k=1$, $p_1=p_2:=p \gg \log n$, $\lambda_{1} = 1 - \sqrt[\leftroot{-3}\uproot{3}4]{\frac{p}{n}}$ and $\lambda_{2} = 1 - 2\sqrt[\leftroot{-3}\uproot{3}4]{\frac{p}{n}}$. Then our bound rates $\frac{(1-\lambda_k^2)(1 - \lambda_{k+1}^2)}{(\lambda_k - \lambda_{k+1})^2}\frac{p_1 - k}{n} + \frac{(p_1 + p_2)^2}{n^2(\lambda_k - \lambda_{k+1})^4} + e^{-\gamma (p_1 \wedge p_2)}$
actually imply that
\[
\E \lossave(\vct{\phi}_1, \widehat{\vct{\phi}}_1)  \leq C\frac{p}{n},
\]
while the rough rates \prettyref{eq:rough_bound} by Wedin's sin-theta law implies
\[
\E \lossave(\vct{\phi}_1, \widehat{\vct{\phi}}_1) \leq C\sqrt{\frac{p}{n}}.
\]
This shows that our upper bound rates could be much sharper than the rough rates \prettyref{eq:rough_bound} when both $\lambda_k$ and $\lambda_{k+1}$ are close to $1$.

\subsubsection*{New proof techniques and connection to asymptotic theory} 
To the best of our knowledge, none of the analysis in \cite{gao2015minimax}, \cite{gao2014sparse}, \cite{cai2016rate} can be used to obtain the multiplicative factor $(1-\lambda_k^2)(1-\lambda_{k+1}^2)/(\lambda_k - \lambda_{k+1})^2$ in the first order term of the upper bound, even under the strong condition that $\lambda_{k+1} = \cdots =\lambda_{p_1\wedge p_2} = 0$.

Following a different path, we do careful non-asymptotic entry-wise perturbation analysis of the estimating equations of CCA to avoid the loss of precision caused by applying matrix inequalities in the early stage of the proof. The main challenge is to analyze the properties of matrix hardmard products, especially to derive tight operator norm bounds for certain hardmard products. We are particularly luckily to find a divide-and-conquer approach ($\lambda_k \geq \frac{1}{2}$ and $\lambda_k < \frac{1}{2}$ in the proof of \prettyref{lem:operator_bound}) to decompose the target matrices into simple-structure matrices where we can apply the tools developed in Lemma~\ref{lem:hada-norm}.

The asymptotic distribution of the canonical loadings $\{(\widehat{\vct{\phi}}_i, \widehat{\vct{\psi}}_i)\}_{i=1}^{p_1\wedge p_2}$ has been studied by \cite{anderson1999asymptotic} under the assumption that all the canonical correlations are distinct and $\lambda_1 \neq 1$. Since we focus on subspaces, we only require $\lambda_k > \lambda_{k+1}$ for the given $k$. Both \cite{anderson1999asymptotic} and our work are based on analyzing the estimating equations (\prettyref{eq:estimate}) of CCA. Our analysis is more involved because completely novel techniques are required to obtain the factor $(1-\lambda_k^2)(1-\lambda_{k+1}^2)$ in the nonasymptotic framework.

\subsection{Sharper lower bounds under parameter spaces with fixed $\lambda_k$ and $\lambda_{k+1}$}
\label{sec:lower_bound}
The minimax lower bounds for the estimation rates of CCA were first established by \cite{gao2015minimax,gao2014sparse} under the losses $\lossjoint$ and $\lossgao$. However, the parameter space discussed in \cite{gao2014sparse} requires $\lambda_{k+1} = 0$. Moreover, the parameter space in \cite{gao2015minimax} is parameterized by $\lambda$ satisfying $\lambda_{k} \geq \lambda$, but $\lambda_{k+1}$ is not specified. In fact, they also constructed the hypothesis class with $\lambda_{k+1}=0$ and the resulting minimax lower bound is proportional to $\frac{1}{\lambda^2}$.

However, this minimax lower bound is not sharp when $\lambda_{k}$ and $\lambda_{k+1}$ are close. Suppose $p_1 = p_2 := p$, $k=1$, $\lambda_1 = \frac{1}{2}$ and $\lambda_2 = \frac{1}{2} - \sqrt{\frac{p}{n}}$. Our minimax lower bound in \prettyref{thm:lower} leads to 
\[
\inf\limits_{\pxhk}\sup_{\s\in  \mathcal{F}}  \E\left[ \lossave(\mtx{\Phi}_{1:k}, \widehat{\mtx{\Phi}}_{1:k})\right]\geq O(1).
\]
In contrast, to capture the fundamental limit of CCA estimates in this scenario under the framework of \cite{gao2015minimax}, one needs to choose $\lambda$ to capture both $\lambda_k$ and $\lambda_{k+1}$, i.e., $\lambda_k \leq \lambda \leq \lambda_{k+1}$ and hence $\lambda \approx 1/2$. Then the resulting minimax lower bound rate will be $\frac{p}{n \lambda^2} = O(\frac{p}{n})$, which is much looser than $O(1)$.

Technically speaking, we follow the analytical framework of \cite{gao2015minimax} and \cite{gao2014sparse}, but the hypothesis classes  construction requires any given $\lambda_{k+1}>0$ instead of $\lambda_{k+1} = 0$, and this brings in new technical challenges. More detailed technical discussions are deferred to \prettyref{sec:lower-proof}.

\strangesection{4}{Proof of \prettyref{thm:formula}}
\label{sec:equivalence-proof}

\label{sec:proofs}

Suppose the observed sample of $(\vct{x}, \vct{y})$ is fixed and consider the correlation between the two subspaces of $\mc{H}$ (defined in \prettyref{eq:hilbert}): $\lspan(U_1, \ldots, U_k)$ and $\lspan(\widehat{U}_1, \ldots, \widehat{U}_k)$. Let $(W_1, \widehat{W}_1), (W_2, \widehat{W}_2), \ldots, (W_k, \widehat{W}_k)$ be the first, second, ..., and  $k$th pair of canonical variates between $U_1, \ldots, U_k$ and $\widehat{U}_1, \ldots, \widehat{U}_k$. Then $\lspan(W_1, \ldots, W_k)=\lspan(U_1, \ldots, U_k)$, $\lspan(\widehat{W}_1, \ldots, \widehat{W}_k)=\lspan(\widehat{U}_1, \ldots, \widehat{U}_k)$ and $\langle W_i, W_j\rangle=\langle W_i, \widehat{W}_j\rangle=\langle \widehat{W}_i, \widehat{W}_j \rangle=0$, for any $i \neq j$ and $\text{Var}(W_i)=\text{Var}(\widehat{W}_i)=1$, for $i=1, \ldots, k$.\\

By the definition of principal angles, we know $\angle(W_i, \widehat{W}_i)$ is actually the $i$th principal angle between $\lspan(U_1, \ldots, U_k)$ and $\lspan(\widehat{U}_1, \ldots, \widehat{U}_k)$, i.e., $\theta_i := \angle (W_i, \widehat{W}_i)$. This implies that
\[
\mathcal{L}_{ave}(\mtx{\Phi}_{1:k}, \widehat{\mtx{\Phi}}_{1:k}):=\sum_{i=1}^k \sin^2 \theta_i = \sum_{i=1}^k \left(1 - \left|\left\langle W_i, \widehat{W}_i\right\rangle\right|^2\right).
\]
Since $U_1, \ldots, U_k, \widehat{U}_1, \ldots, \widehat{U}_k$ are linear combinations of $X_1, \ldots, X_{p_1}$, we can denote
\[
\vct{w}^\top := (W_1, \ldots, W_k)=\vct{x}^\top\mtx{\Sigma}_x^{-\half}\mtx{B}, \text{~and~} \hat{\vct{w}}^\top:=(\widehat{W}_1, \ldots, \widehat{W}_k)=\vct{x}^\top\mtx{\Sigma}_x^{-\half}\widehat{\mtx{B}},
\]
where $\mtx{B}:=[\vct{b}_1, \ldots, \vct{b}_k],~ \widehat{\mtx{B}}:=[\widehat{\vct{b}}_1, \ldots, \widehat{\vct{b}}_k] \in \mathbb{R}^{p \times k}$.

By the definition of $\vct{w}$, we have 
\[
\mtx{I}_k = \text{Cov}(\vct{w}) = \mtx{B}^\top \mtx{\Sigma}_x^{-1/2} \text{Cov}(\vct{x})\mtx{\Sigma}_x^{-1/2} \mtx{B} = \mtx{B}^\top \mtx{B}
\]
and similarly $\mtx{I}_k = \widehat{\mtx{B}}^\top \widehat{\mtx{B}}$. Then $\mtx{B}, \widehat{\mtx{B}}$ are $p \times k$ basis matrices. Moreover, we have $\vct{b}_i^\top\widehat{\vct{b}}_j=\langle W_i, \widehat{W}_j \rangle=0$, for all $i \neq j$. Moreover, we have
\[
\text{Diag}(\cos(\theta_1), \ldots, \cos(\theta_k)) = \text{Cov}(\vct{w}, \hat{\vct{w}}) = \mtx{B}^\top \mtx{\Sigma}_x^{-1/2} \text{Cov}(\vct{x})\mtx{\Sigma}_x^{-1/2} \widehat{\mtx{B}} = \mtx{B}^\top \widehat{\mtx{B}}.
\]
Notice that $\lspan(U_1, \ldots, U_k)=\lspan(W_1, \ldots, W_k)$, $(U_1, \ldots, U_k)= \vct{x}^\top\mtx{\Phi}_{1:k}$, and $(W_1, \ldots, W_k)=\vct{x}^\top\mtx{\Sigma}_x^{-\half}\mtx{B}$. Then 
\[
\mtx{\Phi}_{1:k} = \mtx{\Sigma}_x^{-\half}\mtx{B}\mtx{C} \Rightarrow \mtx{\Sigma}_x^{\half}\mtx{\Phi}_{1:k} = \mtx{B}\mtx{C}
\]
for some nonsingular $k \times k$ matrix $C$. This implies that $\mtx{B}$ and $\mtx{\Sigma}_x^{\half}\mtx{\Phi}_{1:k}$ have the same column space. Since $\mtx{B} \in \mathbb{R}^{p \times k}$ is a basis matrix, we have 
\[
\mtx{B}\mtx{B}^\top = \mtx{P}_{\mtx{\Sigma}_x^{\half}\mtx{\Phi}_{1:k}}.
\]
Similarly, we have 
\[
\widehat{\mtx{B}}\widehat{\mtx{B}}^\top = \mtx{P}_{\mtx{\Sigma}_x^{\half}\widehat{\mtx{\Phi}}_{1:k}}.
\] 
Straightforward calculation gives
\begin{align*}
\left\|\mtx{B}\mtx{B}^\top - \widehat{\mtx{B}}\widehat{\mtx{B}}^\top\right\|_F^2 &= \trace \left(\mtx{B}\mtx{B}^\top\mtx{B}\mtx{B}^\top - \mtx{B}\mtx{B}^\top\widehat{\mtx{B}}\widehat{\mtx{B}}^\top - \widehat{\mtx{B}}\widehat{\mtx{B}}^\top\mtx{B}\mtx{B}^\top + \widehat{\mtx{B}}\widehat{\mtx{B}}^\top\widehat{\mtx{B}}\widehat{\mtx{B}}^\top\right)
\\
&=2k - 2\trace(\mtx{B}^\top\widehat{\mtx{B}}\widehat{\mtx{B}}^\top\mtx{B})
\\
&=2k - 2\trace(\text{Diag}(\cos^2 (\theta_1), \ldots, \cos^2 (\theta_k)))
\\
& = 2 (\sin^2 (\theta_1)+ \ldots+ \sin^2 (\theta_k)) = 2k \mathcal{L}_{ave}(\mtx{\Phi}_{1:k}, \widehat{\mtx{\Phi}}_{1:k}) 
\end{align*}
and
\begin{align*}
\left\| \left(\mtx{I}_{p_1} - \mtx{B}\mtx{B}^\top\right)\widehat{\mtx{B}}\widehat{\mtx{B}}^\top\right\|_F^2 &= \trace \left(\left(\mtx{I}_{p_1} - \mtx{B}\mtx{B}^\top\right)\widehat{\mtx{B}}\widehat{\mtx{B}}^\top \widehat{\mtx{B}}\widehat{\mtx{B}}^\top\left(\mtx{I}_{p_1} - \mtx{B}\mtx{B}^\top\right) \right)
\\
&=k - \trace(\mtx{B}^\top\widehat{\mtx{B}}\widehat{\mtx{B}}^\top\mtx{B})
\\
&= k \mathcal{L}_{ave}(\mtx{\Phi}_{1:k}, \widehat{\mtx{\Phi}}_{1:k}). 
\end{align*}
The above equalities yield the first two equalities in \prettyref{eq:formula1}.

Notice that both $U_1, \ldots, U_k$ and $W_1, \ldots W_k$ are both orthonormal bases of $\lspan(U_1, \ldots, U_k)$. (Similarly, $\widehat{U}_1, \ldots, \widehat{U}_k$ and $\widehat{W}_1, \ldots \widehat{W}_k$ are both orthonormal bases of $\lspan(\widehat{U}_1, \ldots, \widehat{U}_k))$.) Then we have $\vct{u}^\top = \vct{w}^\top \mtx{R}$ where $\mtx{R}$ is a $k \times k$ orthogonal matrix. Then 
\begin{align*}
\min_{\mtx{Q} \in \mathbb{R}^{k \times k}} \E \|\vct{u}^\top - \hat{\vct{u}}^\top \mtx{Q}\|_2^2 &= \min_{\mtx{Q} \in \mathbb{R}^{k \times k}} \E \|\vct{u}^\top - \hat{\vct{w}}^\top \mtx{Q}\|_2^2 = \min_{\mtx{Q} \in \mathbb{R}^{k \times k}} \E \|\vct{w}^\top \mtx{R}- \hat{\vct{w}}^\top \mtx{Q}\|_2^2
\\
&=\min_{\mtx{Q} \in \mathbb{R}^{k \times k}} \E \|\vct{w}^\top - \hat{\vct{w}}^\top \mtx{Q} \mtx{R}^\top\|_2^2 = \min_{\mtx{Q} \in \mathbb{R}^{k \times k}} \E \|\vct{w}^\top - \hat{\vct{w}}^\top \mtx{Q}\|_2^2
\\
&=\min_{\vct{q}_i \in \mathbb{R}^k,~i=1, \ldots, k} \E \sum_{i=1}^k (W_i - \hat{\vct{w}}^\top \vct{q}_i)^2 
\\
&= \min_{\vct{q}_i \in \mathbb{R}^k,~i=1, \ldots, k} \sum_{i=1}^k \E (W_i - \hat{\vct{w}}^\top \vct{q}_i)^2
\\
&=\sum_{i=1}^k\min_{\vct{q}_i \in \mathbb{R}^k}  \E  (W_i - \hat{\vct{w}}^\top \vct{q}_i)^2
\end{align*}
Notice that $\min_{\vct{q}_i \in \mathbb{R}^k}  \E  (W_i - \hat{\vct{w}}^\top \vct{q}_i)^2$ is obtained by the best linear predictor, so 
\begin{align*}
\min_{\vct{q}_i \in \mathbb{R}^k}  \E  (W_i - \hat{\vct{w}}^\top \vct{q}_i)^2 &= \Var(W_i) - \Cov(\hat{\vct{w}}, W_i)^\top \Cov^{-1}(\hat{\vct{w}}) \Cov(\hat{\vct{w}}, W_i)
\\
&= 1 - \cos^2 \theta_i = \sin^2 \theta_i.
\end{align*}
Therefore,
\[
\min_{\mtx{Q} \in \mathbb{R}^{k \times k}} \E \|\vct{u}^\top - \hat{\vct{u}}^\top \mtx{Q}\|_2^2 = \sum_{i=1}^k \sin^2 \theta_i = k\mathcal{L}_{ave}(\mtx{\Phi}_{1:k}, \widehat{\mtx{\Phi}}_{1:k}), 
\]
which implies the third equality in \prettyref{eq:formula1}. Similarly, \\
\begin{align*}
&\max_{\vct{g}\in\mathbb{R}^k, \|\vct{g}\|=1} \min_{\mtx{Q} \in \mathbb{R}^{k \times k}} \E \left( \left(\vct{u}^\top - \hat{\vct{u}}^\top \mtx{Q}\right)\vct{g} \right)^2 
\\
& = \max_{\vct{g}\in\mathbb{R}^k, \|\vct{g}\|=1} \min_{\mtx{Q} \in \mathbb{R}^{k \times k}}  \E \left( \left(\vct{u}^\top - \hat{\vct{w}}^\top \mtx{Q}\right)\vct{g} \right)^2 \\
& = \max_{\vct{g}\in\mathbb{R}^k, \|\vct{g}\|=1} \min_{\mtx{Q} \in \mathbb{R}^{k \times k}}  \E \left( \left(\vct{w}^\top \mtx{R} - \hat{\vct{w}}^\top \mtx{Q}\right)\mtx{R}^\top \vct{g} \right)^2 
\\
& = \max_{\vct{g}\in\mathbb{R}^k, \|\vct{g}\|=1} \min_{\mtx{Q} \in \mathbb{R}^{k \times k}}  \E \left( \left(\vct{w}^\top  - \hat{\vct{w}}^\top \mtx{Q}\right)\vct{g} \right)^2 
\\
&=\max_{\vct{g}\in\mathbb{R}^k, \|\vct{g}\|=1} \min_{\vct{q}_i \in \mathbb{R}^k,~i=1, \ldots, k} \E \sum_{i=1}^k g_i^2(W_i - \hat{\vct{w}}^\top \vct{q}_i)^2 
\\
&= \max_{\vct{g}\in\mathbb{R}^k, \|\vct{g}\|=1} \sum_{i=1}^k g_i^2\sin^2\theta_i 
\\
&=\sin^2\theta_1
\end{align*}

Finally, we prove \prettyref{eq:formula2}. By \cite{wedin1983angles}, we have
\begin{align*}
\left\|\bB\bB^\top-\hbB\hbB^\top\right\|^2 &= \left\| \left(\mtx{I}_{p_1} - \mtx{B}\mtx{B}^\top\right)\widehat{\mtx{B}}\widehat{\mtx{B}}^\top\right\|^2=\left\| \left(\mtx{I}_{p_1} - \mtx{B}\mtx{B}^\top\right)\widehat{\mtx{B}}\right\|^2
\\ 
&=\lambda_{\max}\left(\widehat{\mtx{B}}^\top\left(\mtx{I}_{p_1} - \mtx{B}\mtx{B}^\top\right)^\top\left(\mtx{I}_{p_1} - \mtx{B}\mtx{B}^\top\right)\widehat{\mtx{B}}\right)
\\
&=\lambda_{\max}\left(\mtx{I}_k - \text{Diag}(\cos^2 (\theta_1), \ldots, \cos^2 (\theta_k))\right)
\\
&= 1-\cos^2(\theta_1)=\sin^2(\theta_1)=\mathcal{L}_{max}(\mtx{\Phi}_{1:k}, \widehat{\mtx{\Phi}}_{1:k}),
\end{align*}
which implies the the equalities in \prettyref{eq:formula2}.

\strangesection{5}{Proof of Upper Bound}
\label{sec:upper-proof}
Throughout this proof, we denote $\Delta := \lambda_k - \lambda_{k+1}$.
\subsection{Linear Invariance}
\label{section-reduction}
Without loss of generality, we assume $p_2 \geq p_1:=p$. By the definition of canonical variables, we know that $U_1, \ldots, U_p$ and $V_1, \ldots, V_p$ are only determined by $\lspan(X_1, \ldots, X_{p_1})$ and $\lspan(Y_1, \ldots, Y_{p_2})$. In other words, for any invertible $\mtx{C}_1 \in \mathbb{R}^{p_1 \times p_1}$ and $\mtx{C}_2 \in \mathbb{R}^{p_2 \times p_2}$, the canonical pairs of $(X_1, \ldots, X_{p_1})\mtx{C_1}$ and $(Y_1, \ldots, Y_{p_2})\mtx{C_2}$ are still $(U_1, V_1), \ldots, (U_{p_1}, V_{p_1})$. Therefore, we can consider the following orthonormal bases
\[
U_1, \ldots, U_{p_1} \in \lspan(X_1, \ldots, X_{p_1})
\]
and
\[
V_1, \ldots, V_{p_1}, V_{p_1+1}, \ldots, V_{p_2}  \in \lspan(Y_1, \ldots, Y_{p_2}).
\]
Here $(V_1, \ldots, V_{p_1}, V_{p_1+1}, \ldots, V_{p_2})$ is an orthonormal extension of  $V_1, \ldots, V_{p_1}$. Therefore, we know that $(U_1, V_1), \ldots, (U_{p_1}, V_{p_1})$ are also the the canonical pairs between $U_1, \ldots, U_{p_1}$ and $V_1, \ldots, V_{p_2}$.\\

Similarly, for a fixed sample of the variables of $\vct{x}$ and $\vct{y}$, the sample canonical pairs $(\widehat{U}_1, \widehat{V}_1), \ldots, (\widehat{U}_{p_1}, \widehat{V}_{p_1})$ are also sample canonical pairs of the corresponding sample of $(X_1, \ldots, X_{p_1})\mtx{C_1}$ and $(Y_1, \ldots, Y_{p_2})\mtx{C_2}$. This can be easily seen from the concept of sample canonical variables. For example, $\widehat{U}_1$ and $\widehat{V}_1$ are respectively the linear combinations of $X_1, \ldots, X_{p_1}$ and $Y_1, \ldots, Y_{p_1}$, such that their corresponding sample variance are both $1$ and sample correlation is maximized. If we replace $(X_1, \ldots, X_{p_1})$ and $(Y_1, \ldots, Y_{p_1})$ with $(X_1, \ldots, X_{p_1})\mtx{C_1}$ and $(Y_1, \ldots, Y_{p_2})\mtx{C_2}$ respectively and seek for the first sample canonical pair, the constraints (linear combinations of the two sets of variables and unit sample variances) and the objective (sample correlation is maximized) are the same as before, so $(\widehat{U}_1, \widehat{V}_1)$ is still the answer. Similarly, $(\widehat{U}_1, \widehat{V}_1), \ldots, (\widehat{U}_{p_1}, \widehat{V}_{p_1})$ are the sample canonical pairs of $(X_1, \ldots, X_{p_1})\mtx{C_1}$ and $(Y_1, \ldots, Y_{p_2})\mtx{C_2}$. In particular, they are the sample canonical pairs of $U_1, \ldots, U_{p_1}$ and $V_1, \ldots, V_{p_2}$.\\

The above argument gives the following convenient fact: In order to bound 
\[
\mathcal{L}_{ave/max}(\lspan(\widehat{U}_1, \ldots, \widehat{U}_k), \lspan(U_1, \ldots, U_k))
\]
we can replace $X_1, \ldots, X_{p_1}, Y_1, \ldots, Y_{p_2}$ with $U_1, \ldots, U_{p_1}, V_1, \ldots, V_{p_2}$. In other words, we can assume $\vct{x}$ and $\vct{y}$ satisfy the standard form
\[
\sx=\bI_{p_1}, ~\sy=\bI_{p_2}, ~\sxy=[\mtx{\Lambda}, \mtx{0}_{p_1 \times (p_2 - p_1)}]:= \widetilde{\mtx{\Lambda}}
\]
where $\mtx{\Lambda}=\text{Diag}(\lambda_1, \lambda_2, \ldots, \lambda_{p_1})\in\mathbb{R}^{p_1\times p_1}$. Moreover
\[
\px_{1:p_1}=\bI_{p_1}, ~\py_{1:p_1}=\begin{bmatrix} \mtx{I}_{p_1} \\ \mtx{0}_{(p_2 - p_1) \times p_1} \end{bmatrix},
\]
which implies that
\[
\px_{1:k}=\begin{bmatrix} \mtx{I}_k \\ \mtx{0}_{(p_1 - k) \times k} \end{bmatrix}, ~\py_{1:k}=\begin{bmatrix} \mtx{I}_{k} \\ \mtx{0}_{(p_2 - k) \times k} \end{bmatrix}.
\]

\subsection{Upper Bound Under the Standard Form}
\label{sec:analytic-expression}
Under the standard form, by \prettyref{eq:formula1} and \prettyref{eq:formula2}, we have
\begin{equation}
\label{eq:ave_representation}
\mathcal{L}_{ave}(\lspan(\widehat{U}_1, \ldots, \widehat{U}_k), \lspan(U_1, \ldots, U_k)) = \frac{1}{k}\left\|\left(\mtx{I}_{p_1} - \mtx{P}_{\pxk}\right)\mtx{P}_{\pxhk}\right\|_F^2
\end{equation}
and
\begin{equation}
\label{eq:max_representation}
\mathcal{L}_{max}(\lspan(\widehat{U}_1, \ldots, \widehat{U}_k), \lspan(U_1, \ldots, U_k)) = \left\|\left(\mtx{I}_{p_1} - \mtx{P}_{\pxk}\right)\mtx{P}_{\pxhk}\right\|^2.
\end{equation}

Denote $\pxhk=\begin{bmatrix}\pxhk^u\\ \pxhk^l \end{bmatrix}$ where $\pxhk^u$ and $\pxhk^l$ are the upper $k\times k$ and lower $(p_1-k)\times k$ sub-matrices of $\pxhk$ respectively. Then
\begin{align*}
\left\|\left(\mtx{I}_{p_1} - \mtx{P}_{\pxk}\right)\mtx{P}_{\pxhk}\right\|_F^2 &= \trace\left(\left(\mtx{I}_{p_1} - \mtx{P}_{\pxk}\right)\pxhk(\pxhk^\top \pxhk)^{-1} \pxhk^\top \left(\mtx{I}_{p_1} - \mtx{P}_{\pxk}\right)\right),
\\
\left\|\left(\mtx{I}_{p_1} - \mtx{P}_{\pxk}\right)\mtx{P}_{\pxhk}\right\|^2 &= \lambda_{\max}\left(\left(\mtx{I}_{p_1} - \mtx{P}_{\pxk}\right)\pxhk(\pxhk^\top \pxhk)^{-1} \pxhk^\top \left(\mtx{I}_{p_1} - \mtx{P}_{\pxk}\right)\right)
\end{align*}
Since
\begin{align*}
&\left(\mtx{I}_{p_1} - \mtx{P}_{\pxk}\right)\pxhk(\pxhk^\top \pxhk)^{-1} \pxhk^\top \left(\mtx{I}_{p_1} - \mtx{P}_{\pxk}\right) 
\\
&\preceq \frac{1}{\sigma_k^2(\pxhk)}\left(\mtx{I}_{p_1} - \mtx{P}_{\pxk}\right)\pxhk \pxhk^\top \left(\mtx{I}_{p_1} - \mtx{P}_{\pxk}\right) = \frac{1}{\sigma_k^2(\pxhk)} \begin{bmatrix}\mtx{0}_{k \times k} \\ \pxhk^l \end{bmatrix}\begin{bmatrix}\mtx{0}_{k \times k} & {(\pxhk^l)}^\top \end{bmatrix},
\end{align*}
we have
\begin{align}
\label{eq:frobenius_control}
\left\|\left(\mtx{I}_{p_1} - \mtx{P}_{\pxk}\right)\mtx{P}_{\pxhk}\right\|_F^2 &\leq \trace\left(\frac{1}{\sigma_k^2(\pxhk)} \begin{bmatrix}\mtx{0}_{k \times k} \\ \pxhk^l \end{bmatrix}\begin{bmatrix}\mtx{0}_{k \times k} & \pxhk^\top \end{bmatrix}\right) = \frac{\|\pxhk^l\|_F^2}{\sigma_k^2(\pxhk)},
\end{align}
and
\begin{align}
\label{eq:operator_control}
\left\|\left(\mtx{I}_{p_1} - \mtx{P}_{\pxk}\right)\mtx{P}_{\pxhk}\right\|^2 &\leq \lambda_{\max}\left(\frac{1}{\sigma_k^2(\pxhk)} \begin{bmatrix}\mtx{0}_{k \times k} \\ \pxhk^l \end{bmatrix}\begin{bmatrix}\mtx{0}_{k \times k} & \pxhk^\top \end{bmatrix}\right) = \frac{\|\pxhk^l\|^2}{\sigma_k^2(\pxhk)}.
\end{align}
Therefore, it suffices to give upper bounds of $\|\pxhk^l\|_F^2$ and $\|\pxhk^l\|^2$, as well as a lower bound of $\sigma_k^2(\pxhk)$.

\subsection{Basic bounds}
Recall that
\[
\sx=\bI_{p_1}, ~\sy=\bI_{p_2}, ~\sxy=[\mtx{\Lambda}, \mtx{0}_{p_1 \times (p_2 - p_1)}]:= \widetilde{\mtx{\Lambda}}.
\]
Then
\[
\Cov\left(\begin{bmatrix} \vct{x} \\ \vct{y} \end{bmatrix}\right):=\s =\begin{bmatrix}
\bI_{p_1} & \widetilde{\mtx{\Lambda}}\\
\widetilde{\mtx{\Lambda}}^\top&\bI_{p_2}
\end{bmatrix}
\]
and
\[
\widehat{\Cov}\left(\begin{bmatrix} \vct{x} \\ \vct{y} \end{bmatrix}\right):=\sh =\begin{bmatrix}
\sxh & \sxyh\\
\syxh&\syh
\end{bmatrix}.
\]
Moreover, we can define $\sh_{2p_1}$ as the left upper $(2p_1) \times (2p_1)$ principal submatrix of $\sh$. We can similarly define $\s_{2p_1}$.

\begin{lemma}
\label{lem:estimation_0}
There exist universal constants $\gamma$, $C$ and $C_0$ such that when $n\geq C_0 p_1$, then with probability at least $1-e^{-\gamma p_1}$, the following inequalities hold
\[
\|\s_{2p_1} -\sh_{2p_1}\|, \|\bI_{p_1} - \sh_x\|, \left\|\sxh^{\half}-\bI_{p_1}\right\| \leq C\sqrt{\frac{p_1}{n}}.
\]
\end{lemma}

\begin{proof}~\\
It is obvious that $\|\mtx{\Sigma}_{2p_1}\| \leq 2$. By \prettyref{lem:vershynin}, there exist constants $\gamma$, $C_0$ and $C_1$, such that when $n\geq C_0 p_1$, with probability at least $1-e^{-\gamma p_1}$ there holds
\[
\|\sh_{2p_1} - \s_{2p_1}\|\leq C_1\sqrt{\frac{p_1}{n}}.
\]
As submatrices, we have $\|\mtx{I}_{p_1} - \sh_x\| \leq C_1\sqrt{\frac{p_1}{n}}$. Moreover,
 \[
 \|\bI_{p_1}-\sxh\|=\|(\bI_{p_1} - \sxh^{\half})(\bI_{p_1} + \sxh^{\half})\|\geq \sigma_{\min}(\bI_{p_1} + \sxh^{\half})\|\bI_{p_1} - \sxh^{\half}\|\geq \|\bI_{p_1} - \sxh^{\half}\|,
\]
which implies $\|\bI_{p_1} - \sxh^{\half}\| \leq C_1\sqrt{\frac{p_1 + p_2}{n}}$.
\end{proof}

\begin{lemma}
\label{lem:estimation}
There exist universal constants $c$, $C$ and $C_0$ such that when $n\geq C_0(p_1+p_2)$, then with probability at least $1-e^{-c (p_1 + p_2)}$, the following inequalities hold
\begin{gather*}
\|\s -\sh\|, \|\bI_{p_2} - \sh_y\|, \|\s_{xy} - \sh_{xy}\|, \left\|\syh^{\half}-\bI_{p_2}\right\| \leq C\sqrt{\frac{p_1 + p_2}{n}},\\
\|\hbLambda -\bLambda\| \leq \|\sxh^{-1/2} \sxyh\syh^{-1/2} -\sxy\| \leq C\sqrt{\frac{p_1 + p_2}{n}},\\
\sigma_k^2(\pxhk)\geq \frac{1}{2},~ \|\pxhk\|^2 \leq \frac{3}{2},~ \sigma_k^2(\pyhk)\geq \frac{1}{2},~ \|\pyhk\|^2 \leq \frac{3}{2},\\
\|\pxhk^l\|, ~\|\pyhk^l\| \leq  \frac{C}{\Delta}\sqrt{\frac{p_1 + p_2}{n}},
\end{gather*}
where $\Delta=\lambda_k-\lambda_{k+1}$ is the eigen-gap. 
\end{lemma}
The proof is deferred to \prettyref{sec:proof_key}.

\subsection{Estimating Equations and upper bound of $\|\pxhk^l\|^2$}
In this section, we aim to give a sharp upper bound for $\|\pxhk^l\|^2$. Notice that we have already established an upper bound in \prettyref{lem:estimation}, where Wedin's $\sin \theta$ law plays the essential role. However, this bound is actually too loose for our purpose. Therefore, we need to develop new techniques to sharpen the results.

Recall that $\pxh\in\mathbb{R}^{p_1\times p_1}, \pyh\in\mathbb{R}^{p_2\times p_1}$ consist of the sample canonical coefficients. By definition, the sample canonical coefficients satisfy the following two estimating equations (because $\sxh^{\half}\pxh$ and $\syh^{\half}\pyh$ are left and right singular vectors of $\sxh^{-\half}\sxyh\syh^{-\half}$ respectively), 
\begin{equation}
\label{eq:estimate}
\begin{aligned}
\sxyh\pyh&=\sxh\pxh\hbLambda 
\\
\syxh\pxh&= \syh \pyh \hbLambda .
\end{aligned}
\end{equation}
If we define define
\begin{align}
\bLambda=\begin{bmatrix}
\bLambda_1 & \\
& \bLambda_2 
\end{bmatrix} \in \mathbb{R}^{p_1 \times p_1},~~
\hbLambda=\begin{bmatrix}
\hbLambda_1 & \\
& \hbLambda_2 
\end{bmatrix} \in \mathbb{R}^{p_1 \times p_1},
\label{eq:p1p2}
\end{align}
where $\bLambda_1,~\hbLambda_1$ are $k\times k$ diagonal matrices while $\bLambda_2,~\hbLambda_2$ are $(p_1 - k) \times (p_1 - k)$ diagonal matrices. Then \prettyref{eq:estimate} imply 
\begin{equation}
\label{eq:estimate1}
\begin{aligned}
\sxyh\pyh_{1:k}&=\sxh\pxh_{1:k}\hbLambda_1 
\\
\syxh\pxh_{1:k}&= \syh \pyh_{1:k} \hbLambda_1.
\end{aligned}
\end{equation}

Divide the matrices into blocks, 
\begin{align*}
\sxh=\begin{bmatrix}
\sxh^{11} &\sxh^{12}\\
\sxh^{21}&\sxh^{22}
\end{bmatrix}, ~
\syh=\begin{bmatrix}
\syh^{11} &\syh^{12}\\
\syh^{21}&\syh^{22}
\end{bmatrix}, ~
\sxyh=\begin{bmatrix}
\sxyh^{11} &\sxyh^{12}\\
\sxyh^{21}&\sxyh^{22}
\end{bmatrix},~
\syxh=\begin{bmatrix}
\syxh^{11} &\syxh^{12}\\
\syxh^{21}&\syxh^{22}
\end{bmatrix}
\end{align*}
where $\sxh^{11}, \syh^{11}, \sxyh^{11}, \syxh^{11}$ are $k\times k$ matrices. 
Finally, we define $\pyhk^u\in\mathbb{R}^{k\times k}, \pyhk^l\in\mathbb{R}^{(p_2-k)\times k}$ in the same way as $\pxhk^u, \pxhk^l$. With these blocks, \prettyref{eq:estimate1} can be rewritten as
\begin{align}
\label{eq:est1} \sxyh^{21}\pyhk^u+\sxyh^{22}\pyhk^l&=\sxh^{21}\pxhk^u\hbLambda_1+\sxh^{22}\pxhk^l \hbLambda_1, 
\\
\label{eq:est2} \syxh^{21}\pxhk^u+\syxh^{22}\pxhk^l&=\syh^{21}\pyhk^u\hbLambda_1+\syh^{22}\pyhk^l \hbLambda_1, 
\\
\label{eq:est3} \sxyh^{11}\pyhk^u+\sxyh^{12}\pyhk^l&=\sxh^{11}\pxhk^u\hbLambda_1+\sxh^{12}\pxhk^l \hbLambda_1, 
\\
\label{eq:est4} \syxh^{11}\pxhk^u+\syxh^{12}\pxhk^l&=\syh^{11}\pyhk^u\hbLambda_1+\syh^{12}\pyhk^l \hbLambda_1.
\end{align}
Define the zero-padding of $\mtx{\Lambda}_2$: 
\[
\widetilde{\mtx{\Lambda}}_2 := [\mtx{\Lambda}_2, \mtx{0}] = \mtx{\Sigma}_{xy}^{22}\in \mathbb{R}^{(p_1 - k) \times (p_2 - k)}.
\]
The above equations imply the following lemma:
\begin{lemma}
\label{lem:pillar_final}
The equality \prettyref{eq:estimate1} gives the following result
\begin{align}
\label{eq:key1} \pxhk^l \bLambda_1^2-\bLambda_2^2\pxhk^l &=\bB\pxhk^u+\mtx{R}
\\
\label{eq:key2} & = (\widehat{\mtx{\Sigma}}_{xy}^{21} - \widehat{\mtx{\Sigma}}_{x}^{21} \mtx{\Lambda}_1) \widehat{\mtx{\Psi}}^u_{1:k} \mtx{\Lambda}_1 + \widetilde{\mtx{\Lambda}}_2(\widehat{\mtx{\Sigma}}_{yx}^{21} - \widehat{\mtx{\Sigma}}_{y}^{21} \mtx{\Lambda}_1)\widehat{\mtx{\Phi}}^u_{1:k} + \widetilde{\mtx{R}}
\end{align}
where
\begin{align*}
\bB&:=\sxyh^{21}\bLambda_1+\widetilde{\bLambda}_2 \syxh^{21}-\sxh^{21}\bLambda_1^2 - \widetilde{\bLambda}_2\syh^{21}\bLambda_1,
\\
\widetilde{\mtx{R}}&:=(\sxh^{21}\mtx{R}_1 - \mtx{R}_3)\bLambda_1-\widetilde{\bLambda}_2(\syh^{21}\mtx{R}_2+\mtx{R}_4),
\\
\mtx{R}&:= \widetilde{\mtx{R}} - (\sxyh^{21}-\sxh^{21}\bLambda_1)\mtx{R}_2.
\end{align*}
and
\begin{align*}
\mtx{R}_1&:= \pxhk^u(\hbLambda_1-\bLambda_1)+(\sxh^{11}-\bI_k)\pxhk^u\hbLambda_1+\sxh^{12}\pxhk^l \hbLambda_1-(\sxyh^{11}-\bLambda_1)\pyhk^u-\sxyh^{12}\pyhk^l,
\\
\mtx{R}_2&:= \pyhk^u(\hbLambda_1-\bLambda_1)+(\syh^{11}-\bI_k)\pyhk^u\hbLambda_1+\syh^{12}\pyhk^l \hbLambda_1-(\syxh^{11}-\bLambda_1)\pxhk^u-\syxh^{12}\pxhk^l,
\\
\mtx{R}_3&:= \sxh^{21}\pxhk^u(\hbLambda_1-\bLambda_1) + (\sxh^{22}\pxhk^l \hbLambda_1-\pxhk^l \bLambda_1)-(\sxyh^{22}-\widetilde{\bLambda}_2)\pyhk^l,
\\
\mtx{R}_4&:= \syh^{21}\pyhk^u(\hbLambda_1-\bLambda_1) + (\syh^{22}\pyhk^l \hbLambda_1-\pyhk^l \bLambda_1)-(\syxh^{22}-\widetilde{\bLambda}_2^\top)\pxhk^l.
\end{align*}
\end{lemma}
The proof is deferred to \prettyref{sec:proof_key}.

~\\
By \prettyref{lem:estimation}, one can easily obtain that
\[
\|\mtx{R}_1\|, \|\mtx{R}_2\| \leq C\sqrt{\frac{p_1 + p_2}{n}}.
\]
Recall that
\[
\mtx{R}_3:= \sxh^{21}\pxhk^u(\hbLambda_1-\bLambda_1) + (\sxh^{22}\pxhk^l \hbLambda_1-\pxhk^l \bLambda_1)-(\sxyh^{22}-\widetilde{\bLambda}_2)\pyhk^l
\]
By \prettyref{lem:estimation}, we have 
\[
\|\sxh^{21}\pxhk^u(\hbLambda_1-\bLambda_1)\|\leq C\frac{p_1 + p_2}{n},~~\|(\sxyh^{22}-\widetilde{\bLambda}_2)\pyhk^l\| \leq C\frac{p_1 + p_2}{\Delta n},
\]
and
\begin{align*}
\|\sxh^{22}\pxhk^l \hbLambda_1-\pxhk^l \bLambda_1\| &\leq \|(\sxh^{22} - \mtx{I}_{p_1 - k})\pxhk^l \hbLambda_1 + \pxhk^l (\hbLambda_1 -\bLambda_1)\| 
\\
&\leq \|(\sxh^{22} - \mtx{I}_{p_1 - k})\pxhk^l \hbLambda_1\| + \|\pxhk^l (\hbLambda_1 -\bLambda_1)\| \leq C\frac{p_1 + p_2}{\Delta n}.
\end{align*}
Therefore, we get $\|\mtx{R}_3\| \leq C\frac{p_1 + p_2}{\Delta n}$. Similarly, $\|\mtx{R}_4\| \leq C\frac{p_1 + p_2}{\Delta n}$.\\

Combined with \prettyref{lem:estimation}, we have
\[
\|\widetilde{\mtx{R}}\| = \|(\sxh^{21}\mtx{R}_1 - \mtx{R}_3)\bLambda_1-\widetilde{\bLambda}_2(\syh^{21}\mtx{R}_2+\mtx{R}_4)\| \leq C\frac{p_1 + p_2}{\Delta n}
\]
and
\[
\|\mtx{R}\| \leq \|\widetilde{\mtx{R}}\| + \|\sxyh^{21}-\sxh^{21}\bLambda_1\|\|\mtx{R}_2\| \leq C\frac{p_1 + p_2}{\Delta n}.
\]
The proof of the following lemma is deferred to  \prettyref{sec:proof_key}:
\begin{lemma}
\label{lem:operator_bound}
If $n \geq C_0(p_1 + p_2)$, then with probability $1 - c_0\exp(-\gamma p_1)$,
\[
\|\pxhk^l\| \leq C \left[\sqrt{\frac{p_1(1-\lambda_k^2)(1-\lambda_{k+1}^2)}{n\Delta^2}} + \frac{(p_1 + p_2)}{n \Delta^2}\right].
\]
\end{lemma}

\subsection{Upper bounds of risks}
Notice that the inequality \prettyref{eq:operator_control} yields
\[
\left\|\left(\mtx{I}_{p_1} - \mtx{P}_{\pxk}\right)\mtx{P}_{\pxhk}\right\|^2 \leq \frac{\|\pxhk^l\|^2}{\sigma_k^2(\pxhk)}.
\]
By \prettyref{lem:operator_bound} and \prettyref{lem:estimation}, we know on an event $G$ with probability at least $1- Ce^{-\gamma p_1}$, 
\[
\left\|\left(\mtx{I}_{p_1} - \mtx{P}_{\pxk}\right)\mtx{P}_{\pxhk}\right\|^2 \leq C\left[\frac{p_1(1-\lambda_k^2)(1-\lambda_{k+1}^2)}{n\Delta^2} + \frac{(p_1 + p_2)^2}{n^2 \Delta^4}\right].
\]
Moreover, since $\left\|\left(\mtx{I}_{p_1} - \mtx{P}_{\pxk}\right)\mtx{P}_{\pxhk}\right\|^2 \leq 1$, by \prettyref{eq:max_representation}, we have
 
\begin{align*}
\E \lossmax(\mtx{\Phi}_{1:k}, \widehat{\mtx{\Phi}}_{1:k})
&=\E \left\|\left(\mtx{I}_{p_1} - \mtx{P}_{\pxk}\right)\mtx{P}_{\pxhk}\right\|^2 \leq C\left[\frac{p_1(1-\lambda_k^2)(1-\lambda_{k+1}^2)}{n\Delta^2} + \frac{(p_1 + p_2)^2}{n^2 \Delta^4} + e^{-\gamma p_1}\right].
\end{align*}

Since $\left(\mtx{I}_{p_1} - \mtx{P}_{\pxk}\right)\mtx{P}_{\pxhk}$ is of at most rank-$k$, we have 
\[
\frac{1}{k}\left\|\left(\mtx{I}_{p_1} - \mtx{P}_{\pxk}\right)\mtx{P}_{\pxhk}\right\|_F^2 \leq \left\|\left(\mtx{I}_{p_1} - \mtx{P}_{\pxk}\right)\mtx{P}_{\pxhk}\right\|^2
\]
Then by \prettyref{eq:ave_representation} and the previous inequality, we have
\begin{align*}
\E \lossave(\mtx{\Phi}_{1:k}, \widehat{\mtx{\Phi}}_{1:k})
&=\E \left\|\left(\mtx{I}_{p_1} - \mtx{P}_{\pxk}\right)\mtx{P}_{\pxhk}\right\|^2 
\\
&=\E \frac{1}{k}\left\|\left(\mtx{I}_{p_1} - \mtx{P}_{\pxk}\right)\mtx{P}_{\pxhk}\right\|_F^2 
\\
&\leq \E \left\|\left(\mtx{I}_{p_1} - \mtx{P}_{\pxk}\right)\mtx{P}_{\pxhk}\right\|^2
\\
&\leq C\left[\frac{p_1(1-\lambda_k^2)(1-\lambda_{k+1}^2)}{n\Delta^2} + \frac{(p_1 + p_2)^2}{n^2 \Delta^4} + e^{-\gamma p_1}\right].
\end{align*}

In fact, the factor $p_1$ in the main term can be reduced to $p_1-k$ by similar arguments as done for the operator norm. The Frobenius norm version of \prettyref{lem:operator_bound} is actually much simpler. We omit the proof to avoid unnecessary redundancy and repetition.


%


\subsection{Supporting lemmas in linear algebra and probability} 
\begin{definition}(Hadamard Operator Norm)
\label{def:Hadamard_norm}
For $\bA\in\mathbb{R}^{m\times n}$, define the Hadamard operator norm as 
\[
\norm{\bA}=\sup\left\{\|\bA\circ \bB\|: \|\bB\|\leq 1, \bB\in\mathbb{R}^{m\times n}\right\}
\]
Let $\alpha_1, \cdots, \alpha_m$ and $\beta_1, \cdots, \beta_n$ be arbitrary positive numbers lower bounded by a positive constant $\delta$. 
\end{definition} 
\begin{lemma}
\label{lem:hada-norm}
Let $\{\alpha_i\}_{i=1}^m$ and $\{\beta_i\}_{i=1}^n$ be two sequences of positive numbers.  for any $\mtx{X} \in \mathbb{R}^{m \times n}$, there hold
\begin{equation}
\label{eq:ave_hada}
\left\| \left[\frac{\sqrt{\alpha_i\beta_j}}{\alpha_i + \beta_j}\right] \circ \mtx{X}\right \| \leq \frac{1}{2} \|\mtx{X}\|,
\end{equation}
and
\begin{equation}
\label{eq:max_hada}
\left\| \left[\frac{\min(\alpha_i, \beta_j)}{\alpha_i + \beta_j}\right] \circ \mtx{X}\right \| \leq \frac{1}{2} \|\mtx{X}\|, \quad \left\| \left[\frac{\max(\alpha_i, \beta_j)}{\alpha_i + \beta_j}\right] \circ \mtx{X}\right \| \leq \frac{3}{2} \|\mtx{X}\|.
\end{equation}
\end{lemma} 
\begin{proof}
The proof of \prettyref{eq:ave_hada} can be found in ``Norm Bounds for Hadamard Products and an Arithmetic-Geometric Mean Inequality for Unitarily Invariant Norms" by Horn.\\

Denote 
\[
\mtx{G}_1 =  \left[\frac{\max(\alpha_i, \beta_j)}{\alpha_i + \beta_j}\right], \mtx{G}_2 =  \left[\frac{\min(\alpha_i, \beta_j)}{\alpha_i + \beta_j}\right]
\]

The proof of \prettyref{eq:max_hada} relies on the following two results. 
\begin{lemma}(Theorem 5.5.18 of \cite{hom1991topics})
If $\bA, \bB\in\mathbb{R}^{n\times n}$ and $\bA$ is  positive semidefinite. Then, 
\begin{equation}
\|\bA\circ \bB\|\leq \left(\max_{1\leq i\leq n} \bA_{ii}\right)\|\bB\|, 
\nonumber
\end{equation}
where $\|\cdot\|$ is the operator norm. 
\label{lem:hada}
\end{lemma}
\begin{lemma}(Theorem 3.2 of \cite{mathias1993hadamard})
The symmetric matrix 
$$\Big(\frac{\min(a_i, a_j)}{a_i+a_j}\Big)_{1\leq i, j\leq n}$$
is positive semidefinite if $a_i>0, 1\leq i\leq n$.
\label{lem:max}
\end{lemma}

Define $\gamma_i=\beta_i, 1\leq i\leq n$ and $\gamma_{i}=\alpha_{i-n}, n+1\leq i\leq m+n$. Define $\bM\in\mathbb{R}^{(m+n)\times (m+n)}$ by 
\[
M_{ij}=\frac{\min\{\gamma_i, \gamma_j\}}{\gamma_i+\gamma_j}. 
\]
By \prettyref{lem:max}, $\bM$ is also positive semidefinite. Again, apply \prettyref{lem:hada} and notice that $\mtx{G}_2$ is the lower left sub-matrix of $\bM$, It is easy to obtain
\[
\norm{\mtx{G}_2}\leq \norm{\bM}\leq \frac{1}{2}. 
\]
Finally, since $\mtx{G}_1\circ \bB=\bB-\mtx{G}_2 \circ \bB$ for any $\bB$, we have
\[
\|\mtx{G}_1\circ \bB\|\leq\|\bB\|+\|\mtx{G}_2 \circ \bB\|, 
\]
which implies, 
\[
\norm{\mtx{G}_1}\leq 1+\norm{\mtx{G}_2}\leq \frac{3}{2}. 
\]
\end{proof}

\begin{lemma}(Covariance Matrix Estimation,  Remark 5.40 of \cite{vershynin2010introduction})
\label{lem:vershynin}
Assume $\bA\in \mathbb{R}^{n\times p}$ has independent sub-gaussian random rows with second moment matrix $\bSigma$. Then there exists universal constant $C$ such that for every $t\geq 0$, the following inequality holds with probability at least $1-e^{-ct^2}$,
\begin{equation}
\|\frac{1}{n}\bA^\top\bA-\bSigma\|\leq \max\{\delta, \delta^2\}\|\bSigma\| \quad\quad \delta=C\sqrt{\frac{p}{n}}+\frac{t}{\sqrt{n}}.
\nonumber
\end{equation}
\end{lemma}

\begin{lemma}(Bernstein inequality, Proposition 5.16 of \cite{vershynin2010introduction})
Let $X_1, \cdots, X_n$ be independent centered sub-exponential random variables and $K=\max_{i}\|X_i\|_{\psi_1}$. Then for every $\ba=(\ba_1, \cdots, \ba_n)\in\mathbb{R}^n$ and every $t\geq 0$, we have
\begin{equation}
P\left\{|\sum_{i=1}^n \ba_i X_i|\geq t \right\}\leq 2exp\left\{-c\min\left(\frac{t^2}{K^2\|\ba\|_2^2}, \frac{t}{K\|\ba\|_{\infty}}\right)\right\}.
\nonumber
\end{equation}
\label{Bernstein}
\end{lemma}
\begin{lemma}(Hanson-Wright inequality, Theorem 1.1 of \cite{hw2013})
Let $\bx=(x_1, \cdots, x_p)$ be a random vector with independent components $x_i$ which satisfy $\E x_i=0$ and $\|x_i\|_{\psi_2}\leq K$, Let $\bA\in\mathbb{R}^{p\times p}$. Then there exists universal constant $c$ such that for every $t\geq 0$,
\begin{equation}
P\left\{|\bx^\top \bA\bx-\E\bx^\top \bA\bx|\geq t \right\}\leq 2exp\left\{-c\min\left(\frac{t^2}{K^4\|\bA\|_F^2}, \frac{t}{K^2\|\bA\|}\right)\right\}.
\nonumber
\end{equation}
\label{HW-inequality}
\end{lemma}

\begin{lemma}(Covering Number of the Sphere, Lemma 5.2 of \cite{vershynin2010introduction}). The unit Euclidean sphere $\mathbb{S}^{n-1}$ equipped with the Euclidean metric satisfies for every $\epsilon> 0$ that
\begin{equation}
|\mathcal{N}(\mathbb{S}^{n-1}, \epsilon)|\leq (1+\frac{2}{\epsilon})^n,
\nonumber
\end{equation}
where $\mathcal{N}(\mathbb{S}^{n-1}, \epsilon)$ is the $\epsilon$-net of $\,\mathbb{S}^{n-1}$ with minimal cardinality. 
\label{CoverNum}
\end{lemma}

The following variant of Wedin's $\sin\theta$ law \citep{wedin1972perturbation} is proved in Proposition 1 of  \cite{cai2015optimal}. 
\begin{lemma}
For $\bA, \bE\in \mathbb{R}^{m\times n}$ and $\hbA=\bA+\bE$, define the singular value decompositions of $\bA$ and $\hbA$ as 
\[
\bA=\bU\bD\bV^\top,~ \hbA=\hbU\hbD\hbV^\top.
\]
Then the following perturbation bound holds, 
\[
\left\|\left(\bI-\bP_{\bU_{1:k}}\right)\bP_{\hbU_{1:k}}\right\|=\left\|\bP_{\bU_{1:k}}-\bP_{\hbU_{1:k}}\right\|\leq \frac{2\|\bE\|}{\sigma_{k}(\bA)-\sigma_{k+1}(\bA)},
\]
where $\sigma_k(\bA), \sigma_{k+1}(\bA)$ are the $k_{th}$ and $(k+1)_{th}$ singular values of $\bA$. 
\label{lemma-wedin}
\end{lemma}

\subsection{Proofs of key lemmas}
\label{sec:proof_key}
\subsubsection{Proof of \prettyref{lem:estimation}}

(1) The proof of 
\[
\|\s -\sh\|, \|\bI_{p_2} - \sh_y\|, \|\s_{xy} - \sh_{xy}\|, \left\|\syh^{\half}-\bI_{p_2}\right\| \leq C\sqrt{\frac{p_1 + p_2}{n}}
\]
is exactly the same as that of \prettyref{lem:estimation_0}.\\
\\
(2) Observe that
\begin{align*}
\sxh^{-1/2} \sxyh\syh^{-1/2} -\sxy&=(\bI_{p_1}-\sxh^{1/2})\sxh^{-1/2}\sxyh\syh^{-1/2}\\
&+ \sxh^{1/2}\sxh^{-1/2}\sxyh\syh^{-1/2}(\bI_{p_2}-\syh^{1/2})+(\sxyh-\sxy).
\end{align*}
and $\|\sxh^{-1/2}\sxyh\syh^{-1/2}\|=\hlambda_1\leq 1$. Then
\begin{align*}
\|\sxh^{-1/2} \sxyh\syh^{-1/2} -\sxy\|&\leq  \|\bI_{p_1}-\sxh^{1/2}\|+ \|\sxh\|\|\bI_{p_2}-\syh^{1/2}\|+\|\sxyh-\sxy\|. 
\end{align*}
Notice that $\hbLambda$ and $\bLambda$ are singular values of $\sxh^{-1/2} \sxyh\syh^{-1/2}$ and $\sxy$ respectively. Hence by the famous Weyl's inequality for singular values,
\begin{align*}
\|\hbLambda -\bLambda\| &\leq \|\sxh^{-1/2} \sxyh\syh^{-1/2} -\sxy\| 
\\
& \leq \|\bI_{p_1}-\sxh\|+ \|\sxh\|\|\bI_{p_2}-\syh^{1/2}\| + \|\sxyh-\sxy\|
\\
& \leq \left(3 + C_1\sqrt{\frac{p_1 + p_2}{n}}\right)C_1\sqrt{\frac{p_1 + p_2}{n}} \leq C_2\sqrt{\frac{p_1 + p_2}{n}}.
\end{align*}

(3) Since $\sxh^{\half}\pxh$ are left singular vectors of $\sxh^{-\half}\sxyh\syh^{-\half}$, we have $\|\sxh^{\half}\pxh\|=1$, $\pxh^\top\sxh\pxh=\bI_{p_1}$ and $\pxh^\top\pxh-\bI_{p_1}=-\pxh^\top(\sxh-\bI_{p_1})\pxh$. Then we have, 
\begin{align*}
\|\pxh^\top\pxh-\bI_{p_1}\|&= \|\pxh^\top(\sxh-\bI_{p_1})\pxh\|\leq \|\pxh^\top\sxh^{1/2}\|\|\sxh^{-1/2}(\sxh-\bI_{p_1})\sxh^{-1/2}\|\|\sxh^{1/2}\pxh\|\\
&=\|\sxh^{-1/2}(\sxh-\bI_{p_1})\sxh^{-1/2}\|. 
\end{align*}
As a submatrix, 
\begin{align*}
\|\pxhk^\top\pxhk-\bI_k\|&\leq \|\sxh^{-1/2}(\sxh-\bI_{p_1})\sxh^{-1/2}\| \leq \|\sxh^{-1}\| \|\sxh-\bI_{p_1}\|\\
&\leq \frac{1}{1-\|\sxh-\bI_{p_1}\|}\|\sxh-\bI_{p_1}\| \leq \frac{\|\sh-\s\|}{1-\|\sh-\s\|} \leq \frac{1}{2}
\end{align*}
as long as $n \geq C_0 (p_1 + p_2)$ for sufficiently large $C_0$. In this case,
\begin{align*}
\sigma_k^2(\pxhk)\geq 1/2, ~ \|\pxhk\|^2\leq 3/2. 
\end{align*}
By the same argument, 
\[
\sigma_k^2(\pyhk)\geq 1/2, ~ \|\pyhk\|^2\leq 3/2.
\]

(4) Recall that 
\[
\px_{1:k}=\begin{bmatrix} \mtx{I}_k \\ \mtx{0}_{(p_1 - k) \times k} \end{bmatrix}, ~\py_{1:k}=\begin{bmatrix} \mtx{I}_{k} \\ \mtx{0}_{(p_2 - k) \times k} \end{bmatrix}.
\]
The last inequality in the lemma relies on the fact that $\sxh^{\half}\pxhk$ and $\px_{1:k}$ are leading $k$ singular vectors of $\sxh^{-1/2} \sxyh\syh^{-1/2}$ and $\sxy$ respectively. By a variant of Wedin's $\sin\theta$ law as stated in Lemma~\ref{lemma-wedin}, 
\begin{align*}
\left\|\bP_{\sxh^{\half}\pxhk}(\bI_{p_1}-\bP_{\px_{1:k}})\right\|\leq \frac{2\|\sxh^{-1/2} \sxyh\syh^{-1/2} -\sxy\|}{\Delta} \leq \frac{2C_2}{\Delta} \sqrt{\frac{p_1 + p_2}{n}}.
\end{align*}
On the other hand, 
\begin{align*}
\left\|\bP_{\sxh^{\half}\pxhk}(\bI_{p_1}-\bP_{\px_{1:k}})\right\|&=\left\|\sxh^{\half}\pxhk(\sxh^{\half}\pxhk)^\top(\bI_{p_1}-\bP_{\px_{1:k}})\right\|\\
&=\left\|(\sxh^{\half}\pxhk)^\top(\bI_{p_1}-\bP_{\px_{1:k}})\right\|\\
&=\left\|(\sxh^{\half}\pxhk)^l\right\|,
\end{align*}
Here the second equality is due to the fact that $\sxh^{\half}\pxhk$ has orthonormal columns. Moreover, $(\sxh^{\half}\pxhk)^l$ denotes the lower $(p_1-k)\times k$ sub-matrix of $\sxh^{\half}\pxhk$. Again, by triangle inequality, 
\begin{align*}
\left\|\pxhk^l\right\|&=\left\|(\sxh^{\half}\pxhk)^l-\left((\sxh^{\half}-\bI_{p_1})\pxhk\right)^l\right\|\\
&\leq \left\|(\sxh^{\half}\pxhk)^l\right\|+\left\|(\sxh^{\half}-\bI_{p_1})\right\| \left\|\pxhk\right\|\\
&\leq \frac{2C_2}{\Delta} \sqrt{\frac{p_1 + p_2}{n}} + \sqrt{\frac{3}{2}}C_1\sqrt{\frac{p_1 + p_2}{n}} \leq \frac{C_3}{\Delta} \sqrt{\frac{p_1 + p_2}{n}}.
\end{align*}
The last inequality is due to $\Delta \leq 1$. Let $C = \max(C_1, C_2, C_3)$, the proof is done.

\subsubsection{Proof of \prettyref{lem:pillar_final}}
The equality \prettyref{eq:est3} implies
\begin{align}
\bLambda_1\pyhk^u- \pxhk^u\bLambda_1&=  \pxhk^u(\hbLambda_1-\bLambda_1)+(\sxh^{11}-\bI_k)\pxhk^u\hbLambda_1+\sxh^{12}\pxhk^l \hbLambda_1\nonumber\\
&~~~-(\sxyh^{11}-\bLambda_1)\pyhk^u-\sxyh^{12}\pyhk^l:=\mtx{R}_1.  \label{eq:delta1}
\end{align}
Similarly, \prettyref{eq:est4} implies
\begin{align}
\bLambda_1\pxhk^u- \pyhk^u\bLambda_1&=  \pyhk^u(\hbLambda_1-\bLambda_1)+(\syh^{11}-\bI_k)\pyhk^u\hbLambda_1+\syh^{12}\pyhk^l \hbLambda_1 \nonumber\\
&~~~-(\syxh^{11}-\bLambda_1)\pxhk^u-\syxh^{12}\pxhk^l:=\mtx{R}_2 .\label{eq:delta2}
\end{align}
The equality \prettyref{eq:est1} is equivalent to
\begin{align*}
\sxyh^{21}\pyhk^u+\widetilde{\bLambda}_2\pyhk^l+(\sxyh^{22}-\widetilde{\bLambda}_2)\pyhk^l&=\sxh^{21}\pxhk^u\bLambda_1+\sxh^{21}\pxhk^u(\hbLambda_1-\bLambda_1)\\
&+\pxhk^l \bLambda_1 +(\sxh^{22}\pxhk^l \hbLambda_1-\pxhk^l \bLambda_1), 
\end{align*}
which can be written as 
\begin{align}
\sxyh^{21}\pyhk^u+\widetilde{\bLambda}_2\pyhk^l&-\sxh^{21}\pxhk^u\bLambda_1-\pxhk^l \bLambda_1 =\sxh^{21}\pxhk^u(\hbLambda_1-\bLambda_1)\nonumber\\
&+(\sxh^{22}\pxhk^l \hbLambda_1-\pxhk^l \bLambda_1)-(\sxyh^{22}-\widetilde{\bLambda}_2)\pyhk^l:=\mtx{R}_3. \label{eq:pillar7}
\end{align}
Apply the same argument to \prettyref{eq:est2}, we obtain
\begin{align}
\syxh^{21}\pxhk^u+\widetilde{\bLambda}_2^\top\pxhk^l&-\syh^{21}\pyhk^u\bLambda_1-\pyhk^l \bLambda_1 =\syh^{21}\pyhk^u(\hbLambda_1-\bLambda_1)\nonumber\\
&+(\syh^{22}\pyhk^l \hbLambda_1-\pyhk^l \bLambda_1)-(\syxh^{22}-\widetilde{\bLambda}_2^\top)\pxhk^l:=\mtx{R}_4. \label{eq:pillar8}
\end{align}
Consider $\eqref{eq:pillar7}\times (-\bLambda_1)-\widetilde{\bLambda}_2\times \eqref{eq:pillar8}$, then 
\begin{align*}
\pxhk^l \bLambda_1^2 - {\bLambda}_2^2\pxhk^l +\sxh^{21}\pxhk^u\bLambda_1^2&-\sxyh^{21}\pyhk^u\bLambda_1-\widetilde{\bLambda}_2\syxh^{21}\pxhk^u+\widetilde{\bLambda}_2\syh^{21}\pyhk^u\bLambda_1\\
&=-(\mtx{R}_3\bLambda_1+\widetilde{\bLambda}_2\mtx{R}_4), 
\end{align*}
that is 
\begin{align}
\pxhk^l \bLambda_1^2 - {\bLambda}_2^2\pxhk^l =&\sxyh^{21}\pyhk^u\bLambda_1+\widetilde{\bLambda}_2\syxh^{21}\pxhk^u\nonumber\\
&-\sxh^{21}\pxhk^u\bLambda_1^2-\widetilde{\bLambda}_2\syh^{21}\pyhk^u\bLambda_1-(\mtx{R}_3\bLambda_1+\widetilde{\bLambda}_2\mtx{R}_4).
\label{eq-upper-target}
\end{align}
Combined with \prettyref{eq:delta1} and \prettyref{eq:delta2}, 
\begin{align}
\pxhk^l \bLambda_1^2 - {\bLambda}_2^2\pxhk^l &=\sxyh^{21}\pyhk^u\bLambda_1+\widetilde{\bLambda}_2\syxh^{21}\pxhk^u-\sxh^{21}\bLambda_1\pyhk^u\bLambda_1+\sxh^{21}\mtx{R}_1\bLambda_1\nonumber\\
&~~~ -\widetilde{\bLambda}_2\syh^{21}\bLambda_1\pxhk^u-\widetilde{\bLambda}_2\syh^{21}\mtx{R}_2-(\mtx{R}_3\bLambda_1+\widetilde{\bLambda}_2\mtx{R}_4)  \nonumber\\
&=(\sxyh^{21}-\sxh^{21}\bLambda_1)\pyhk^u\bLambda_1+\widetilde{\bLambda}_2(\syxh^{21}-\syh^{21}\bLambda_1)\pxhk^u  \nonumber\\
&~~~+(\sxh^{21}\mtx{R}_1 - \mtx{R}_3)\bLambda_1-\widetilde{\bLambda}_2(\syh^{21}\mtx{R}_2+\mtx{R}_4).
\label{eq:target_upper1}
\end{align}
This finishes the proof of \prettyref{eq:key2}. 

Plug \prettyref{eq:delta2} into \prettyref{eq:target_upper1}, we get
\begin{align*}
\pxhk^l \bLambda_1^2-\widetilde{\bLambda}_2^2\pxhk^l &=(\sxyh^{21}-\sxh^{21}\bLambda_1)(\bLambda_1\pxhk^u- \mtx{R}_2) +\widetilde{\bLambda}_2(\syxh^{21}-\syh^{21}\bLambda_1)\pxhk^u + \widetilde{\mtx{R}} \nonumber
\\
&=\bB\pxhk^u + (\widetilde{\mtx{R}} - (\sxyh^{21}-\sxh^{21}\bLambda_1)\mtx{R}_2).
\end{align*}
This finishes the proof of \prettyref{eq:key1}.

\subsubsection{Proof of \prettyref{lem:operator_bound}}
First, we discuss two quite different cases: $\lambda_k \geq \frac{1}{2}$ and $\lambda_k < \frac{1}{2}$.
\subsection*{Case 1: $\lambda_k \geq \frac{1}{2}$}
Let 
\[
\delta := \lambda_k^2 - \lambda_{k+1}^2 = (\lambda_k - \lambda_{k+1})(\lambda_k + \lambda_{k+1}) \geq \frac{1}{2} \Delta.
\]
Define the $(p_1-k)\times k$ matrices $\bA$ by 
\[
A_{ij}=\frac{\sqrt{\lambda_j^2 - \lambda_k^2 + \frac{\delta}{2}}\sqrt{\lambda_{k+1}^2 - \lambda_{k+i}^2 + \frac{\delta}{2}}}{\lambda_j^2 - \lambda_{k+i}^2},~ 1\leq i\leq p_1-k, 1\leq j\leq k
\]
By \prettyref{eq:key1} in \prettyref{lem:pillar_final}, there holds
\begin{align*}
\pxhk^l = \bA \circ (\mtx{D}_1 \bB\pxhk^u \mtx{D}_2) + \bA \circ (\mtx{D}_1\mtx{R}\mtx{D}_2), 
\end{align*}
where 
\[
\bD_1=\text{diag}\left(\frac{1}{\sqrt{\frac{\delta}{2}}}, \cdots, \frac{1}{\sqrt{\lambda^2_{k+1} - \lambda^2_{p_1} + \frac{\delta}{2}}}\right) 
\]
and
\[
\bD_2=\text{diag}\left(\frac{1}{\sqrt{\lambda_1^2 - \lambda_k^2 + \frac{\delta}{2}}}, \cdots, \frac{1}{\sqrt{\frac{\delta}{2}}}\right). 
\]
By \prettyref{lem:hada-norm}, we have
\begin{align*}
\|\pxhk^l\| &\leq \frac{1}{2}\|\mtx{D}_1 \bB\pxhk^u \mtx{D}_2\| + \frac{1}{2}\|(\mtx{D}_1\mtx{R}\mtx{D}_2)\|
\\
&\leq \frac{1}{2}\|\bD_1 \bB \| \|\pxhk^u\| \| \mtx{D}_2\| + \frac{1}{2}\|\mtx{D}_1\|\|\mtx{R}\|\|\mtx{D}_2\|.
\end{align*}
Recall that $\|\pxhk^u\| \leq \|\pxhk\| \leq \sqrt{\frac{3}{2}}$ and it is obvious that $\|\mtx{D}_1\|, \|\mtx{D}_2\| \leq \sqrt{\frac{2}{\delta}}$. Moreover, in the previous section, we also have shown that $\|\mtx{R}\| \leq \frac{C(p_1 + p_2)}{n \Delta}$. It suffices to bound $\|\mtx{D}_1\mtx{B}\|$ and to this end we apply the standard covering argument.

\textbf{Step 1. Reduction. } Denote by $\mathcal{N}_\epsilon(\mathbb{S}^d)$ the $d$-dimensional unit ball surface. For $\epsilon>0$ and any pair of vectors $\bu\in\mathbb{R}^{p_1-k}, \bv\in\mathbb{R}^k$, we can choose $\bu_{\epsilon}\in \mathcal{N}_\epsilon(\mathbb{S}^{p_1-k-1}),  \bv_{\epsilon}\in \mathcal{N}_\epsilon(\mathbb{S}^{k-1})$ such that $\|\bu-\bu_{\epsilon}\|, \|\bv-\bv_{\epsilon}\|\leq \epsilon$. Then
\begin{align*}
\bu^\top \bD_1\bB \bv&=\bu^\top \bD_1 \bB \bv - \bu_\epsilon^\top \bD_1\bB \bv+\bu_\epsilon^\top \bD_1\bB \bv-\bu_\epsilon^\top \bD_1\bB \bv_\epsilon + \bu_\epsilon^\top \bD_1\bB \bv_\epsilon\\
&\leq \|\bu-\bu_{\epsilon}\|\| \bD_1 \bB \bv\|+\|\bu_\epsilon^\top\bD_1 \bB \|\|\bv-\bv_{\epsilon}\|+\bu_\epsilon^\top \bD_1 \bB \bv_\epsilon\\
&\leq 2\epsilon\|\bD_1\bB \|+\bu_\epsilon^\top \bD_1\bB \bv_\epsilon\\
&\leq 2\epsilon\|\bD_1 \bB \|+\max\limits_{\bu_{\epsilon}, \bv_{\epsilon}}\bu_\epsilon^\top\bD_1 \bB \bv_\epsilon. 
\end{align*}
Maximize over $\bu$ and $\bv$, we obtain
\begin{equation*}
\|\bD_1 \bB \|\leq 2\epsilon\|\bD_1 \bB \|+\max\limits_{\bu_\epsilon, \bv_\epsilon}\bu_\epsilon^\top \bD_1 \bB \bv_\epsilon. 
\end{equation*}
Therefore, $\|\bD_1\bB \|\leq (1-2\epsilon)^{-1}\max\limits_{\bu_\epsilon, \bv_\epsilon}\bu_\epsilon^\top \bD_1 \bB \bv_\epsilon$. Let $\epsilon=1/4$. Then it suffices to give an upper bound $\max\limits_{\bu_\epsilon, \bv_\epsilon}\bu_\epsilon^\top\bD_1 \bB\bv_\epsilon$ with high probability.\\

\textbf{Step 2. Concentration. }
Let $Z_{\alpha, l}=\frac{Y_{\alpha, l} - \lambda_l X_l}{\sqrt{1-\lambda_l^2}}$ for all $1\leq \alpha \leq n$ and $1\leq l\leq p_1$. Then for $1\leq i\leq p_1-k$ and $1\leq j\leq k$
\begin{align*}
&[\bD_1 \bB ]_{i, j}
\\
&=\frac{1}{\sqrt{\lambda_{k+1}^2 - \lambda_{k+i}^2 + \frac{\delta}{2}}} \frac{1}{n}\sum\limits_{\alpha=1}^n (\lambda_j X_{\alpha, k+i} Y_{\alpha, j}- \lambda_j^2 X_{\alpha, k+i} X_{\alpha, j}+ \lambda_{k+i} Y_{\alpha, k+i} X_{\alpha, j}  - \lambda_{k+i}\lambda_j Y_{\alpha, k+i}Y_{\alpha, j})\\
&=\frac{1}{\sqrt{\lambda_{k+1}^2 - \lambda_{k+i}^2 + \frac{\delta}{2}}} \frac{1}{n}\sum\limits_{\alpha=1}^n \Big\{(1-\lambda_j^2)\lambda_{k+i}\lambda_j X_{\alpha, k+i}X_{\alpha, j} -\lambda_j^2(Y_{\alpha, k+i}-\lambda_{k+i}X_{\alpha, k+i})(Y_{\alpha, j}-\lambda_jX_{\alpha, j})\\
&~~~+(1-\lambda_j^2)\lambda_j(Y_{\alpha, k+i}-\lambda_{k+i}X_{\alpha, k+i})X_{\alpha, j}+(1-\lambda_j^2)\lambda_{k+i}(Y_{\alpha, j}-\lambda_jX_{\alpha, j})X_{\alpha, k+i}\Big\}.
\\
&=\frac{1}{\sqrt{\lambda_{k+1}^2 - \lambda_{k+i}^2 + \frac{\delta}{2}}} \frac{1}{n}\sum\limits_{\alpha=1}^n \Big\{(1-\lambda_j^2)\lambda_{k+i}\lambda_j X_{\alpha, k+i}X_{\alpha, j} \\
&~~~ -\lambda_j^2\sqrt{1-\lambda_{k+i}^2}\sqrt{1-\lambda_j^2}Z_{\alpha, k+i}Z_{\alpha, j} + (1-\lambda_j^2)\lambda_j\sqrt{1-\lambda_{k+i}^2}Z_{\alpha, k+i}X_{\alpha, j}
\\
&~~~ +(1-\lambda_j^2)\lambda_{k+i}\sqrt{1-\lambda_{k+i}^2}X_{\alpha, k+i} Z_{\alpha, j}\Big\}. 
\end{align*}

In this way, $\{X_{\alpha, k+i}, Z_{\alpha, k+i}, 1\leq i\leq p_1, 1\leq \alpha\leq n\}$ are mutually independent standard gaussian random variables. For any given pair of vectors $\bu\in\mathbb{R}^{p_1-k}, \bv\in\mathbb{R}^k$, 
\begin{align*}
\bu^\top\bD_1\bB\bv&=\frac{1}{n}\sum_{\alpha=1}^n \sum_{i=1}^{p_1 - k}\sum_{j=1}^k \frac{u_i v_j}{\sqrt{\lambda_{k+1}^2 - \lambda_{k+i}^2 + \frac{\delta}{2}}}\Big\{(1-\lambda_j^2)\lambda_{k+i}\lambda_j X_{\alpha, k+i}X_{\alpha, j}
\\
&~~~ -\lambda_j^2\sqrt{1-\lambda_{k+i}^2}\sqrt{1-\lambda_j^2}Z_{\alpha, k+i}Z_{\alpha, j}+(1-\lambda_j^2)\lambda_j\sqrt{1-\lambda_{k+i}^2}Z_{\alpha, k+i}X_{\alpha, j}
\\
&~~~ +(1-\lambda_j^2)\lambda_{k+i}\sqrt{1-\lambda_{k+i}^2} X_{\alpha, k+i} Z_{\alpha, j} \Big\}\\
&\doteq \frac{1}{n}\sum_{\alpha=1}^n \bw_{\alpha}^\top \mtx{A}_\alpha \bw_{\alpha}, 
\end{align*}
where
\[
\vct{w}_\alpha^\top = [\vct{x}_\alpha^\top, \vct{z}_\alpha^\top] = [X_{\alpha, 1}, \ldots, X_{\alpha, p_1}, Z_{\alpha, 1}, \ldots, Z_{\alpha, p_1}]
\]
and $\mtx{A}_\alpha \in \mathbb{R}^{(2p_1) \times (2p_1)}$ is symmetric and determined by the corresponding quadratic form. This yields
\begin{align*}
\|\bA_\alpha\|_F^2&= \frac{1}{2} \sum_{i=1}^{p_1 - k}\sum_{j=1}^k \frac{u_i^2 v_j^2}{\lambda_{k+1}^2 - \lambda_{k+i}^2 + \frac{\delta}{2}}\Big\{(1-\lambda_j^2)^2\lambda_{k+i}^2\lambda_j^2 +\lambda_j^4(1-\lambda_{k+i}^2)(1-\lambda_j^2)\\
&~~~+(1-\lambda_j^2)^2\lambda_j^2(1-\lambda_{k+i}^2)+(1-\lambda_j^2)^2\lambda_{k+i}^2(1-\lambda_{k+i}^2) \Big\}\\
&=\frac{1}{2} \sum_{i=1}^{p_1 - k}\sum_{j=1}^k \frac{u_i^2 v_j^2}{\lambda_{k+1}^2 - \lambda_{k+i}^2 + \frac{\delta}{2}}\left(1-\lambda_j^2\right)\left(\lambda_{k+i}^2+\lambda_j^2-2\lambda_{k+i}^2\lambda_j^2\right)
\\
&\leq \frac{1}{2} \left(\sum_{i=1}^{p_1 - k} u_i^2\right) \left(\sum_{j=1}^k v_j^2\right) \max_{\substack{1\leq i\leq p_1-k \\1\leq j\leq k}}\frac{(1-\lambda_j^2)(\lambda_{k+i}^2+\lambda_j^2-2\lambda_{k+i}^2\lambda_j^2)}{\lambda_{k+1}^2 - \lambda_{k+i}^2 + \frac{\delta}{2}} \\
&\leq \frac{1}{2} \max_{\substack{1\leq i\leq p_1-k \\1\leq j\leq k}}\frac{(1-\lambda_k^2)(2\lambda_j^2-2\lambda_{k+i}^2\lambda_j^2)}{\lambda_{k+1}^2 - \lambda_{k+i}^2 + \frac{\delta}{2}} \\
&\leq (1-\lambda_k^2)\max_{\substack{1\leq i\leq p_1-k \\1\leq j\leq k}}\frac{\lambda_j^2(1-\lambda_{k+i}^2)}{\frac{\delta}{2} +\lambda_{k+1}^2 - \lambda_{i + k}^2} \\
&\leq (1-\lambda_k^2) \max_{\substack{1\leq i\leq p_1-k \\1\leq j\leq k}} \frac{(1-\lambda_{k+1}^2)}{\frac{\delta}{2}} \\
&\leq \frac{2(1-\lambda_k^2)(1-\lambda_{k+1}^2)}{\delta}\doteq K^2, 
\end{align*} 
where the second last inequality is due to the facts that $\lambda_j \leq 1$ and  
\[
\frac{(1-\lambda_{k+i}^2)}{\frac{\delta}{2} +\lambda_{k+1}^2 - \lambda_{i + k}^2} \leq \frac{(1-\lambda_{k+1}^2)}{\frac{\delta}{2}}~~~ (\because \frac{\delta}{2} +\lambda_{k+1}^2 < \lambda_k^2 \leq 1).
\]
Moreover $\|\bA_\alpha \|_2^2\leq \|\bA_\alpha \|_F^2 \leq K^2$.\\

Now define $\vct{w}^\top := [\vct{w}_1^\top, \ldots, \vct{w}_n^\top]$ and 
\[
\mtx{A} = \begin{bmatrix} \mtx{A}_1 & ~ & ~ & ~ \\ ~ & \mtx{A}_2 & ~ & ~ \\ ~ & ~ & \ddots & ~ \\ ~ & ~ & ~ & \mtx{A}_n \end{bmatrix}.
\]
Then we have 
\[
\|\mtx{A}\| \leq \max_{1 \leq \alpha \leq n} \|\mtx{A}_\alpha\| \leq K, \quad \|\mtx{A}\|_F^2 \leq \sum_{\alpha = 1}^n \|\mtx{A}_\alpha\|_F^2 \leq nK^2
\]
and
\[
\bu^\top\bD_1\bB\bv = \frac{1}{n} \vct{w}^\top \mtx{A} \vct{w}, \text{~where~} \vct{w} \in \mathcal{N}_{2p_1n} (\vct{0}, \mtx{I}_{2p_1n}).
\]

Therefore, By the classic Hanson-Wright inequality (Lemma~\ref{HW-inequality}), there holds
\[
P\left\{ n | \bu^\top\bD_1\bB\bv |\geq t \right\}\leq 2\exp\left\{-c_0\min\left(\frac{t^2}{n K^2}, \frac{t}{K}\right)\right\} 
\]
for some numerical constant $c_0>0$. Without loss of generality, we can also assume $c_0 \leq 1$. Let $t = \frac{4}{c_0} \sqrt{n p_1} K$. By $n \geq p_1$, straightforward calculation gives
\[
P\left\{ n | \bu^\top\bD_1\bB\bv |\geq  \frac{4}{c_0} \sqrt{n p_1} K \right\}\leq 2e^{-4p_1}. 
\]

%

\textbf{Step 3. Union Bound. } By Lemma~\ref{CoverNum}, we choose $1/4$-net such that 
\begin{align*}
&P\left\{\max\limits_{\substack{\bu_\epsilon \in \mathcal{N}_\epsilon(\mathbb{S}^{p_1-k-1}) \\ \bv_\epsilon \in \mathcal{N}_\epsilon(\mathbb{S}^{k-1})}}\bu_\epsilon^\top \bD_1 \bB \bv_\epsilon \geq \left(\frac{4\sqrt{2}}{c_0}\right)\sqrt{\frac{p_1}{n}} \sqrt{\frac{(1-\lambda_k^2)(1-\lambda_{k+1}^2)}{\delta}}\right\}
\\
&\leq 9^{p_1-k}9^{k}\times 2 e^{-4p_1} \leq 2 e^{-\frac{3}{2}p_1}.
\end{align*}
In other words, with probability at least $1 - 2 e^{-\frac{3}{2}p_1}$, we have
\[
\|\bD_1\bB \|\leq (1-2\epsilon)^{-1}\max\limits_{\bu_\epsilon, \bv_\epsilon}\bu_\epsilon^\top \bD_1 \bB \bv_\epsilon \leq \left(\frac{8\sqrt{2}}{c_0}\right)\sqrt{\frac{p_1}{n}} \sqrt{\frac{(1-\lambda_k^2)(1-\lambda_{k+1}^2)}{\delta}}.
\]

In summary, we have as long as $n \geq C_0(p_1 + p_2)$, with probability $1 - c_0\exp(-\gamma p_1)$,
\begin{align*}
\|\pxhk^l\| &\leq C \left[\sqrt{\frac{p_1(1-\lambda_k^2)(1-\lambda_{k+1}^2)}{n\delta^2}} + \frac{(p_1 + p_2)}{n \Delta \delta}\right]
\\
&\leq C \left[\sqrt{\frac{p_1(1-\lambda_k^2)(1-\lambda_{k+1}^2)}{n\Delta^2}} + \frac{(p_1 + p_2)}{n \Delta^2}\right].
\end{align*}
Here the last inequality is due to $\delta = (\lambda_k + \lambda_{k+1})\Delta \geq \frac{1}{2} \Delta$. Here $C_0, C, c_0, \gamma$ are absolute constants.

\subsection*{Case 2: $\lambda_k \leq \frac{1}{2}$}~\\
By \prettyref{eq:key2}, we have
\[
\pxhk^l \bLambda_1^2-\bLambda_2^2\pxhk^l = \mtx{G} \mtx{\Lambda}_1 + \mtx{\Lambda}_2 \mtx{F},
\]
where
\[
\mtx{G} := (\widehat{\mtx{\Sigma}}_{xy}^{21} - \widehat{\mtx{\Sigma}}_{x}^{21} \mtx{\Lambda}_1) \widehat{\mtx{\Psi}}^u_{1:k} + (\sxh^{21}\mtx{R}_1 - \mtx{R}_3)
\]
and
\[
\mtx{F}:= [\mtx{I}_{p_1}, \mtx{0}_{p_1 \times (p_2 -p_1)}]\left[(\widehat{\mtx{\Sigma}}_{yx}^{21} - \widehat{\mtx{\Sigma}}_{y}^{21} \mtx{\Lambda}_1)\widehat{\mtx{\Phi}}^u_{1:k} - (\syh^{21}\mtx{R}_2+\mtx{R}_4)\right].
\]
Notice that $\widehat{\mtx{\Sigma}}_{xy}^{21}$ and $\widehat{\mtx{\Sigma}}_{x}^{21}$ are submatrices of $\sh_{2p_1}$. By \prettyref{lem:estimation_0}, we have
\[
\|\widehat{\mtx{\Sigma}}_{xy}^{21} - \widehat{\mtx{\Sigma}}_{x}^{21} \mtx{\Lambda}_1\| \leq C\sqrt{\frac{p_1}{n}}.
\]
Moreover, by $\|\mtx{R}_1\| \leq C\sqrt{\frac{p_1 + p_2}{n}}$, $\|\mtx{R}_3\| \leq C\frac{p_1 + p_2}{n \Delta}$ and \prettyref{lem:estimation}, there holds 
\[
\|\mtx{G}\| \leq C\left(\sqrt{\frac{p_1}{n}} + \frac{p_1 + p_2}{n \Delta}\right).
\]
Similarly, $[\mtx{I}_{p_1}, \mtx{0}_{p_1 \times (p_2 -p_1)}] \widehat{\mtx{\Sigma}}_{yx}^{21}$ and $[\mtx{I}_{p_1}, \mtx{0}_{p_1 \times (p_2 -p_1)}] \widehat{\mtx{\Sigma}}_{x}^{21}$ are submatrices of $\sh_{2p_1}$. By a similar argument, 
\[
\|\mtx{F}\| \leq C\left(\sqrt{\frac{p_1}{n}} + \frac{p_1 + p_2}{n \Delta}\right).
\]

Then
\[
\pxhk^l = \left[\frac{\lambda_j}{\lambda_{k+i} + \lambda_j}\right] \circ \left[\frac{1}{\lambda_j - \lambda_{k+i}}\right] \circ \mtx{G} + \left[\frac{\lambda_{k+i}}{\lambda_{k+i} + \lambda_j}\right] \circ \left[\frac{1}{\lambda_j - \lambda_{k+i}}\right] \circ \mtx{F}
\]
Here $1\leq i\leq p_1-k$ and $1\leq j\leq k$. By \prettyref{lem:hada-norm}, there holds for any $\mtx{X}$,
\[
\left\| \left[\frac{\lambda_j}{\lambda_{k+i} + \lambda_j}\right] \mtx{X}\right\| = \left\| \left[\frac{\max(\lambda_{k+i}, \lambda_j)}{\lambda_{k+i} + \lambda_j}\right] \mtx{X}\right\| \leq \frac{3}{2}\|\mtx{X}\|
\]
and
\[
\left\| \left[\frac{\lambda_{k+i}}{\lambda_{k+i} + \lambda_j}\right] \mtx{X}\right\| = \left\| \left[\frac{\min(\lambda_{k+i}, \lambda_j)}{\lambda_{k+i} + \lambda_j}\right] \mtx{X}\right\| \leq \frac{1}{2}\|\mtx{X}\|.
\]
Finally, for any $\mtx{X}$, 
\[
\left[\frac{1}{\lambda_j - \lambda_{k+i}}\right] \mtx{X}= \bA \circ (\mtx{D}_1 \mtx{X} \mtx{D}_2)
\]
where
\[
\bA:= \left[\frac{\sqrt{\lambda_j - \lambda_k + \frac{\Delta}{2}} \sqrt{\lambda_{k+1} - \lambda_{k+i} + \frac{\Delta}{2}}}{\lambda_j - \lambda_{k+i}}\right],
\]
\[
\bD_1=\text{diag}\left(\frac{1}{\sqrt{\frac{\Delta}{2}}}, \cdots, \frac{1}{\sqrt{\lambda_{k+1} - \lambda_{p_1} + \frac{\Delta}{2}}}\right), 
\]
and
\[
\bD_2=\text{diag}\left(\frac{1}{\sqrt{\lambda_1 - \lambda_k + \frac{\Delta}{2}}}, \cdots, \frac{1}{\sqrt{\frac{\Delta}{2}}}\right). 
\]
Since $\|\bD_1\|, \|\bD_2\| \leq \sqrt{\frac{2}{\Delta}}$, by \prettyref{lem:hada-norm},
\[
\left\| \left[\frac{1}{\lambda_j - \lambda_{k+i}}\right] \mtx{X}\right\| \leq \frac{1}{2} \|\mtx{D}_1 \mtx{X} \mtx{D}_2\| \leq \frac{1}{\Delta} \|\mtx{X}\|.
\]
In summary, we have
\[
\|\pxhk^l\| \leq C\left(\sqrt{\frac{p_1}{n \Delta^2}} + \frac{p_1 + p_2}{n \Delta^2}\right).
\]
Since $\frac{1}{2} \geq \lambda_k \geq \lambda_{k+1}$, there holds
\[
\|\pxhk^l\| \leq C \left[\sqrt{\frac{p_1(1-\lambda_k^2)(1-\lambda_{k+1}^2)}{n\Delta^2}} + \frac{(p_1 + p_2)}{n \Delta^2}\right].
\]

\strangesection{6}{Lower Bound: Proof of Theorem~\ref{thm:lower}}
\label{sec:lower-proof}
\label{sec::lower-proof}
To establish the minimax lower bounds of CCA estimates for our proposed losses, we follow the analytical frameworks in the literature of PCA and CCA, e.g., \cite{vu2013minimax, cai2013sparse, gao2015minimax}, where the calculation is focused on the construction of the hypothesis class to which the packing lemma and Fano's inequality are applied. However, since we fix both $\lambda_k$ and $\lambda_{k+1}$ in the localized parameter spaces, new technical challenges arise and consequently we construct hypothesis classes based on the equality \eqref{eq::kl-construction}. In this section we also denote $\Delta := \lambda_k - \lambda_{k+1}$.


\subsection{On Kullback-Leibler Divergence} 
The following lemma can be viewed as an extension of Lemma 14 in \cite{gao2015minimax} from $\lambda_{k+1} = 0$ to arbitrary $\lambda_{k+1}$. The proof of the lemma can be found in Section~\ref{sec::KL}. 
\begin{lemma}
For $i=1, 2$ and $p_2\geq p_1\geq k$, let $\begin{bmatrix}
\bU_{(i)}, ~\bW_{(i)}
\end{bmatrix}\in \mc{O}(p_1, p_1), \begin{bmatrix}
\bV_{(i)}, ~\bZ_{(i)}
\end{bmatrix}\in \mc{O}(p_2, p_1)$ where $\bU_{(i)}\in\mathbb{R}^{p_1\times k}, \bV_{(i)}\in\mathbb{R}^{p_2\times k}$. For  $0\leq \lambda_2<\lambda_1<1$, let $\Delta=\lambda_1-\lambda_2$ and define
\[
\s_{(i)}=\begin{bmatrix} \sx & \sx^{1/2}(\lambda_1\bU_{(i)}\bV_{(i)}^\top+\lambda_2\bW_{(i)}\bZ_{(i)}^\top)\sy^{1/2}\\ \sy^{1/2}(\lambda_1\bV_{(i)}\bU_{(i)}^\top+\lambda_2\bZ_{(i)}\bW_{(i)}^\top) \sx^{1/2}& \sy \end{bmatrix}~i=1,2,
\]
 Let $\P_{(i)}$ denote the distribution of a random $i. i. d. $ sample of size $n$ from $N(0, \s_{(i)})$. If we further assume 
 \begin{align}
 [\bU_{(1)}, \bW_{(1)}]\begin{bmatrix}
\smallskip
\bV_{(1)}^\top\\
 \bZ_{(1)}^\top
\end{bmatrix}=[\bU_{(2)}, \bW_{(2)}]\begin{bmatrix}
\smallskip
\bV_{(2)}^\top\\
 \bZ_{(2)}^\top
\end{bmatrix},
\label{eq::kl-construction}
 \end{align}
Then one can show that
\[
D(\P_{(1)}||\P_{(2)})=\frac{n\Delta^2(1+\lambda_1\lambda_2)}{2(1-\lambda_1^2)(1-\lambda_2^2)}\|\bU_{(1)}\bV_{(1)}^\top-\bU_{(2)}\bV_{(2)}^\top\|_F^2.
\]
\label{KL-distance}
\end{lemma}

\begin{remark}
	The conditon in \eqref{eq::kl-construction} is crucial for obtaining the eigen-gap factor $1/\Delta^2$ in the lower bound and is the key insight behind the construction of the hypothesis class in the proof. \cite{gao2015minimax} has a similar lemma but only deals with the case that the residual canonical correlations are zero.  To the best of our knowledge, the proof techniques in \cite{gao2015minimax,gao2014sparse} cannot be directly used to obtain our results. 
\end{remark}


\subsection{Packing Number and Fano's Lemma}
The following result on the packing number is based on the metric entropy of the Grassmannian manifold $G(k,r)$ due to \cite{szarek1982nets}. We use the version adapted from Lemma 1 of \cite{cai2013sparse} which is also used in \cite{gao2015minimax}. 
\begin{lemma}
For any fixed $\bU_0\in \mc{O}(p, k)$ and $\mc{B}_{\epsilon_0}=\{\bU\in \mc{O}(p, k): \|\bU\bU^\top-\bU_0\bU_0^\top\|_F\leq \epsilon_0 \}$ with $\epsilon_0\in (0, \sqrt{2[k\wedge (p-k)]}\,)$. Define the semi-metric $\rho(\cdot, \cdot)$ on $\mc{B}_{\epsilon_0}$ by 
\[
\rho(\bU_1, \bU_2)=\|\bU_1\bU_1^\top - \bU_2\bU_2^\top\|_F.
\]
Then there exists universal constant $C$ such that for any $\alpha\in (0, 1)$, the packing number $\mathcal{M}(\mc{B}_{\epsilon_0}, \rho, \alpha\epsilon_0)$ satisfies
\[
\mathcal{M}(\mc{B}_{\epsilon_0}, \rho, \alpha\epsilon_0)\geq \left(\frac{1}{C\alpha}\right)^{k(p-k)}.
\]
\label{lemma-packing-szarek}
\end{lemma}
The following corollary is used to prove the lower bound. 
\begin{corollary}
If we change the set in Lemma~\ref{lemma-packing-szarek} to $\wt{\mc{B}}_{\epsilon_0}=\{\bU\in \mc{O}(p, k): \|\bU-\bU_0\|_F\leq \epsilon_0 \}$, then we still have 
\[
\mathcal{M}(\wt{\mc{B}}_{\epsilon_0}, \rho, \alpha\epsilon_0)\geq \left(\frac{1}{C\alpha}\right)^{k(p-k)}.
\]
\label{corollary-packing}
\end{corollary}
\begin{proof}
Apply Lemma~\ref{lemma-packing-szarek} to $\mc{B}_{\epsilon_0}$, there exists $\bU_1, \cdots, \bU_{n}$ with $n\geq \left(1/C\alpha\right)^{k(p-k)}$ such that 
\[
\|\bU_{i}\bU_{i}^\top-\bU_0\bU_0^\top\|_F\leq \epsilon_0,~ 1\leq i\leq n, ~
\|\bU_{i}\bU_{i}^\top-\bU_j\bU_j^\top\|_F\geq \alpha\epsilon_0, 1\leq i\leq j\leq n. 
\]
Define $\wt{\bU}_i=\arg\min\limits_{\bU\in\{\bU_i\bQ, ~\bQ\in\mc{O}(k)\}} \|\bU-\bU_0\|_F$, by Lemma~\ref{subspace-metric}, 
\[
\|\wt{\bU}_i-\bU_0\|_F\leq \|\wt{\bU}_i\wt{\bU}_i^\top-\bU_0\bU_0^\top\|_F\leq \epsilon_0.
\]
Therefore, $\wt{\bU}_1, \cdots, \wt{\bU}_{n}\in\wt{\mc{B}}_{\epsilon_0}$ and 
\[
\|\wt{\bU}_i\wt{\bU}_i^\top-\wt{\bU}_j\wt{\bU}_j^\top\|_F=\|\bU_{i}\bU_{i}^\top-\bU_j\bU_j^\top\|_F\geq \alpha\epsilon_0.
\]
which implies, 
\[
\mathcal{M}(\wt{\mc{B}}_{\epsilon_0}, \rho, \alpha\epsilon_0)\geq n\geq \left(\frac{1}{C\alpha}\right)^{k(p-k)}. 
\]
\end{proof}

\begin{lemma}
For any matrices $\bU_1, \bU_2\in \mc{O}(p, k)$,
\begin{equation}
\inf\limits_{\bQ\in \mc{O}(k, k)}\|\bU_1-\bU_2\bQ\|_F\leq \|\bP_{\bU_1}-\bP_{\bU_2}\|_F
\nonumber
\end{equation}
\label{subspace-metric}
\end{lemma}

\begin{proof} By definition
	\begin{align*}
\|\bU_1-\bU_2\bQ\|_F^2=2k-2tr(\bU_1^\top\bU_2\bQ)
\end{align*}
Let $\bU_1^\top\bU_2=\bU\bD\bV^\top$ be the singular value decomposition. Then $\bV\bU^\top\in O(k, k)$ and 
\begin{align*}
\inf\limits_{\bQ\in O(k, k)}\|\bU_1-\bU_2\bQ\|_F^2&\leq 2k-2tr(\bU_1^\top\bU_2\bV\bU^\top)\\\
&=2k-2tr(\bU\bD\bU^\top)\\
&=2k-2tr(\bD).
\end{align*}
On the other hand, 
\begin{align*}
\|\bP_{\bU_1}-\bP_{\bU_2}\|_F^2&=\|\bU_1\bU_1^\top-\bU_2\bU_2^\top\|_F^2\\
&=2k-2tr(\bU_1\bU_1^\top\bU_2\bU_2^\top)\\
&=2k-2tr(\bU_1^\top\bU_2\bU_2^\top\bU_1)\\
&=2k-2tr(\bD^2).
\end{align*}
Since $\bU_1, \bU_2\in O(p, k)$, $\|\bU_1^\top\bU_2\|\leq 1$ and therefore all the diagonal elements of $\bD$ is less than 1, which implies that $tr(\bD)\geq tr(\bD^2)$ and
\begin{equation}
\inf\limits_{\bQ\in O(k, k)}\|\bU_1-\bU_2\bQ\|_F^2\leq \|\bP_{\bU_1}-\bP_{\bU_2}\|_F^2.
\nonumber
\end{equation}
\end{proof}

\begin{lemma}[Fano's Lemma \cite{yu1997assouad}] Let $(\Theta, \rho)$ be a (semi)metric space and $\{\P_{\theta}: \theta\in\Theta\}$ a collection of probability measures. For any totally bounded $T\subset\Theta$, denote $\mathcal{M}(T, \rho, \epsilon)$ the $\epsilon$-packing number of $T$ with respect to the metric $\rho$, i.e. , the maximal number of points in $T$ whoese pairwise minimum distance in $\rho$ is at least $\epsilon$. Define the Kullback-Leibler diameter of $T$ by 
\[
d_{\text{KL}}(T)=\sup\limits_{\theta, \theta^\prime\in T} D(\P_{\theta}||\P_{\theta^\prime}). 
\]
Then, 
\[
\inf\limits_{\htheta}\sup\limits_{\theta\in\Theta} \E_{\theta} \left[\rho^2 (\htheta, \theta) \right]\geq \sup\limits_{T\subset\Theta}\sup\limits_{\epsilon>0} \frac{\epsilon^2}{4}\Big(1-\frac{d_{\text{KL}}(T)+\text{log }2}{\text{log } \mathcal{M}(T, \rho, \epsilon)}\Big)
\]
\label{Lemma-exp-Fano}
\end{lemma}

\subsection{Proof of Lower Bound}
For any fixed $\begin{bmatrix}
\bU_{(0)}, ~\bW_{(0)}
\end{bmatrix}\in\mathcal{O}(p_1, p_1)$ and  $ \begin{bmatrix}
\bV_{(0)}, ~\bZ_{(0)}
\end{bmatrix}\in \mathcal{O}(p_2, p_1)$ where $\bU_{(0)}\in\mathbb{R}^{p_1\times k}, \bV_{(0)}\in\mathbb{R}^{p_2\times k}, \bW_{(0)}\in\mathbb{R}^{p_1\times (p_1-k)}, \bV_{(0)}\in\mathbb{R}^{p_2\times (p_2-k)}$, define
\begin{align*}
\mathcal{H}_{\epsilon_0}=\Big\{\big(\bU, \bW, \bV, \bZ\big): ~ \begin{bmatrix}
\bU, ~\bW
\end{bmatrix}\in \mc{O}(p_1, p_1)~\text{with}~ \bU\in \mathbb{R}^{p_1\times k},~ \begin{bmatrix}
\bV, ~\bZ
\end{bmatrix}\in \mc{O}(p_2, p_1)\\~\text{with}~\bV\in \mathbb{R}^{p_2\times k}
,\|\bU-\bU_{(0)}\|_F\leq \epsilon_0, ~
[\bU, \bW]\begin{bmatrix}
\smallskip
\bV^\top\\
 \bZ^\top
\end{bmatrix}=[\bU_{(0)}, \bW_{(0)}]\begin{bmatrix}
\smallskip
\bV_{(0)}^\top\\
 \bZ_{(0)}^\top
\end{bmatrix}
 \Big\}. 
\end{align*}
For any fixed $\sx\in\mathbb{S}^{p_1}_+, \sy\in\mathbb{S}^{p_2}_+$ with $\kappa(\sx)=\kappa_x, \kappa(\sy)=\kappa_y$, consider the parametrization $\s_{xy}=\sx\px\Lambda\py^\top\sy$, for   $0\leq \lambda_{k+1}<\lambda_k<1$, define
\begin{align*}
\mathcal{T}_{\epsilon_0}=\Big\{
\s=\begin{bmatrix} \sx & \sx^{1/2}(\lambda_k\bU\bV^\top+\lambda_{k+1}\bW\bZ^\top)\sy^{1/2}\\ \sy^{1/2}(\lambda_k\bV\bU^\top+\lambda_{k+1}\bZ\bW^\top) \sx^{1/2}& \sy \end{bmatrix},\\
 \px=\sx^{-1/2}[\bU,\bW],  \py=\sy^{-1/2}[\bV, \bZ], \big(\bU, \bW, \bV, \bZ\big)\in\mathcal{H}_{\epsilon_0}  \Big\}.
\end{align*}
It is straightforward to verify that $\mathcal{T}_{\epsilon_0}\subset \mathcal{F}(p_1, p_2, k, \lambda_k,\lambda_{k+1},\kappa_x, \kappa_y)$. For any $\s_{(i)}\in\mathcal{T}_{\epsilon_0}, ~i=1, 2$, they yield to the parametrization, 
\begin{align*}
\s_{(i)}=\begin{bmatrix} \sx & \sx^{1/2}(\lambda_k\bU_{(i)}\bV_{(i)}^\top+\lambda_{k+1}\bW_{(i)}\bZ_{(i)}^\top)\sy^{1/2}\\ \sy^{1/2}(\lambda_k\bV_{(i)}\bU_{(i)}^\top+\lambda_{k+1}\bZ_{(i)}\bW_{(i)}^\top) \sx^{1/2}& \sy \end{bmatrix},
\end{align*}
where $\big(\bU_{(i)}, \bW_{(i)}, \bV_{(i)}, \bZ_{(i)}\big)\in\mathcal{H}_{\epsilon_0}$ and the leading-$k$ canonical vectors are $\pxk^{(i)}=\sx^{-\half}\bU_{(i)}, \pyk^{(i)}=\sy^{-\half}\bV_{(i)}$. We define a semi-metric on $\mathcal{T}_{\epsilon_0}$ as 
\[
\rho(\s_{(1)}, \s_{(2)})=\left\|\bP_{\sx^{\half}\pxk^{(1)}}-\bP_{\sx^{\half}\pxk^{(2)}}\right\|_F=\left\| \bP_{\bU_{(1)}}-\bP_{\bU_{(2)}}\right\|_F.
\]
By Lemma~\ref{KL-distance}, 
\begin{align*}
D(\P_{\s_1}||\P_{\s_2})=\frac{n\Delta^2(1+\lambda_k\lambda_{k+1})}{2(1-\lambda_k^2)(1-\lambda_{k+1}^2)}\|\bU_{(1)}\bV_{(1)}^\top-\bU_{(2)}\bV_{(2)}^\top\|_F^2.
\end{align*}
Further by the definition of $d_{KL} (T)$,
\begin{align}
d_{KL} (T)&=\frac{n\Delta^2(1+\lambda_k\lambda_{k+1})}{2(1-\lambda_k^2)(1-\lambda_{k+1}^2)}\sup\limits_{\s_{(1)}, \s_{(2)}\in \mathcal{T}_{\epsilon_0}}\|\bU_{(1)}\bV_{(1)}^\top-\bU_{(2)}\bV_{(2)}^\top\|_F^2.
\label{eq-KL-obj}
\end{align}
To bound the Kullback-Leibler diameter, for any $\s_{(1)}, \s_{(2)}\in \mathcal{T}_{\epsilon_0}$, by definition, 
\[
[\bU_{(1)}, \bW_{(1)}]\begin{bmatrix}
\smallskip
\bV_{(1)}^\top\\
 \bZ_{(1)}^\top
\end{bmatrix}=[\bU_{(2)}, \bW_{(2)}]\begin{bmatrix}
\smallskip
\bV_{(2)}^\top\\
 \bZ_{(2)}^\top
\end{bmatrix},
\]
which implies that they are singular value decompositions of the same matrix. Therefore, there exists $\bQ\in\mathcal{O}(p_1, p_1)$ such that 
\begin{align}
[\bU_{(2)}, \bW_{(2)}]=[\bU_{(1)}, \bW_{(1)}]\bQ~,~[\bV_{(2)},  \bZ_{(2)}]=[\bV_{(1)},  \bZ_{(1)}]\bQ.
\label{eq-UV}
\end{align}
Decompose $\bQ$ into four blocks such that 
\[
\bQ=\begin{bmatrix}
\bQ_{11}& \bQ_{12}\\
\bQ_{21} & \bQ_{22}
\end{bmatrix}.
\]
Substitute into \eqref{eq-UV}, 
\[
\bU_{(2)}=\bU_{(1)}\bQ_{11}+\bW_{(1)}\bQ_{21},~ \bV_{(2)}=\bV_{(1)}\bQ_{11}+\bZ_{(1)}\bQ_{21}.
\]
Then, 
\begin{align*}
\|\bU_{(2)}-\bU_{(1)}\|_F^2&=\|\bU_{(1)}(\bQ_{11}-\bI_{k})+\bW_{(1)}\bQ_{21}\|_F^2\\
&=\|\bU_{(1)}(\bQ_{11}-\bI_{k})\|_F^2+\|\bW_{(1)}\bQ_{21}\|_F^2\\
&=\|\bQ_{11}-\bI_{k}\|_F^2+\|\bQ_{21}\|_F^2.
\end{align*}
The second equality is due to the fact that $\bU_{(1)}$ and $\bW_{(1)}$ have orthogonal column space and the third equality is valid because $\bU_{(1)}, \bW_{(1)}\in\mathcal{O}(p_1, k)$. By the same argument, we will have 
\[
\|\bV_{(2)}-\bV_{(1)}\|_F^2=\|\bQ_{11}-\bI_{k}\|_F^2+\|\bQ_{21}\|_F^2.
\]
Notice that 
\begin{align*}
\|\bU_{(1)}\bV_{(1)}^\top-\bU_{(2)}\bV_{(2)}^\top\|_F^2&=\|(\bU_{(1)}-\bU_{(2)})\bV_{(1)}+\bU_{(2)}(\bV_{(1)}-\bV_{(2)})\|_F^2\\
&\leq 2\|\bU_{(1)}-\bU_{(2)}\|_F^2+2\|\bV_{(1)}-\bV_{(2)}\|_F^2\\
&=4\|(\bU_{(1)}-\bU_{(2)})\|_F^2\\
&\leq 8\left(\|(\bU_{(1)}-\bU_{(0)})\|_F^2+\|(\bU_{(0)}-\bU_{(2)})\|_F^2\right)\\
&\leq 16\epsilon_0^2.
\end{align*}
Then, substitute into \eqref{eq-KL-obj}
\begin{align}
d_{KL} (T)\leq \frac{8n\Delta^2(1+\lambda_k\lambda_{k+1})}{(1-\lambda_k^2)(1-\lambda_{k+1}^2)}\epsilon_0^2.
\label{eq-KL-bound}
\end{align}
Let $\mathcal{B}_{\epsilon_0}=\{\bU\in O(p_1, k): \|\bU-\bU_{(0)}\|_F\leq \epsilon_0  \}$. Under the semi-metric $\wt{\rho}(\bU_{(1)}, \bU_{(2)})=\|\bU_{(1)}\bU_{(1)}^\top - \bU_{(2)}\bU_{(2)}^\top\|_F$,  we claim that the packing number of $\mathcal{H}_{\epsilon_0}$ is lower bounded by the packing number of $\mathcal{B}_{\epsilon_0}$. To prove this claim, it suffices to show that for any $\bU\in\mathcal{B}_{\epsilon_0}$, there exists corresponding $\bW, \bV, \bZ$ such that $\left(\bU, \bW, \bV, \bZ\right)\in \mathcal{H}_{\epsilon_0}$. First of all, by definition, $\|\bU-\bU_{0}\|_F\leq \epsilon_0$. Let $\bW\in\mathcal{O}(p_1, p_1-k)$ be the orthogonal complement of $\bU$. Then $[\bU, \bW]\in\mathcal{O}(p_1, p_1)$ and therefore there exists $\bQ\in \mathcal{O}(p_1, p_1)$ such that 
\[
[\bU, \bW]=[\bU_{(0)}, \bW_{(0)}]\bQ.
\]
Set $[\bV, \bZ]=[\bV_{(0)}, \bZ_{(0)}]\bQ\in \mathcal{O}(p_2, p_1)$, then 
\[
[\bU, \bW]\begin{bmatrix}
\smallskip
\bV^\top\\
 \bZ^\top
\end{bmatrix}=[\bU_{(0)}, \bW_{(0)}]\begin{bmatrix}
\smallskip
\bV_{(0)}^\top\\
 \bZ_{(0)}^\top
\end{bmatrix},
\]
which implies $\left(\bU, \bW, \bV, \bZ\right)\in \mathcal{H}_{\epsilon_0}$. Let 
\[
\epsilon=\alpha\epsilon_0=c\left(\sqrt{k\wedge(p_1-k)}\wedge\sqrt{\frac{(1-\lambda_k^2)(1-\lambda_{k+1}^2)}{n\Delta^2(1+\lambda_k\lambda_{k+1})}k(p_1-k)}\right),
\]
where $c\in (0, 1)$ depends on $\alpha$ and is chosen small enough such that $\epsilon_0=\epsilon/\alpha\in (0, \sqrt{2[k\wedge (p_1-k)]}\,)$. By Corollary~\ref{corollary-packing}, 
\[
\mathcal{M}(\mathcal{T}_{\epsilon_0}, \rho, \alpha\epsilon_0)=\mathcal{M}(\mathcal{H}_{\epsilon_0}, \wt{\rho}, \alpha\epsilon_0)\geq \mathcal{M}(\mathcal{B}_{\epsilon_0}, \wt{\rho}, \alpha\epsilon_0)\geq \left(\frac{1}{C\alpha}\right)^{k(p_1-k)}.
\]
Apply Lemma~\ref{Lemma-exp-Fano} with $\mathcal{T}_{\epsilon_0}, \rho, \epsilon$, 
\[
\inf\limits_{\pxhk}\sup_{\s\in  \mathcal{F}}  \E\left[~\left\| \bP_{\sx^{\half}\pxhk}-\bP_{\sx^{\half}\pxk}\right\|_F^2 \right]\geq \sup\limits_{T\subset\Theta}\sup\limits_{\epsilon>0} \frac{\epsilon^2}{4}\left(1- \frac{8c^2k(p_1-k)+\text{log} 2}{k(p_1-k)\text{log}\frac{1}{C\alpha}}\right).
\]
Choose $\alpha$ small enough such that 
\[
1- \frac{8c^2k(p_1-k)+\text{log} 2}{k(p_1-k)\text{log}\frac{1}{C\alpha}}\geq \frac{1}{2}.
\]
Then the lower bound is reduced to
\begin{align*}
\inf\limits_{\pxhk}\sup_{\s\in  \mathcal{F}}  \E\left[~ \left\| \bP_{\sx^{\half}\pxhk}-\bP_{\sx^{\half}\pxk}\right\|_F^2 \right]&\geq \frac{c^2}{8} \left\{\frac{(1-\lambda_k^2)(1-\lambda_{k+1}^2)}{n\Delta^2(1+\lambda_k\lambda_{k+1})}k(p_1-k)\wedge k\wedge(p_1-k)\right\}\\
&\geq C^2k\left\{\left(\frac{(1-\lambda_{k}^2)(1-\lambda_{k+1}^2)}{\Delta^2}\frac{p_1-k}{n}\right)\wedge 1\wedge \frac{p_1-k}{k}\right\}
\end{align*}
By symmetry, 
\[
\inf\limits_{\pyhk}\sup_{\s\in  \mathcal{F}}  \E\left[~ \left\| \bP_{\sy^{\half}\pyhk}-\bP_{\sy^{\half}\pyk}\right\|_F^2 \right]\geq C^2k\left\{\left(\frac{(1-\lambda_{k}^2)(1-\lambda_{k+1}^2)}{\Delta^2}\frac{p_1-k}{n}\right)\wedge 1\wedge \frac{p_1-k}{k}\right\}
\]
The lower bound for operator norm error can be immediately obtained by noticing that $\bP_{\sy^{\half}\pyhk}-\bP_{\sy^{\half}\pyk}$ has at most rank $2k$ and 
\[
\left\| \bP_{\sx^{\half}\pxhk}-\bP_{\sx^{\half}\pxk}\right\|^2\geq \frac{1}{2k}\left\| \bP_{\sx^{\half}\pxhk}-\bP_{\sx^{\half}\pxk}\right\|_F^2
\]

\subsection{Proof of Lemma~\ref{KL-distance}}
\label{sec::KL}
By simple algebra, the Kullback-Leibler divergence between two multivariate gaussian distributions satisfies
\[
D(\P_{\s_{(1)}}||\P_{\s_{(2)}})=\frac{n}{2}\Big\{\text{Tr}\left(\s_{(2)}^{-1}(\s_{(1)}-\s_{(2)})\right)-\text{log det}(\s_{(2)}^{-1}\s_{(1)})\Big\}.
\]
Notice that 
\[
\s_{(i)}=\begin{bmatrix}
\sx^{\half} & \\
&\sy^{\half}
\end{bmatrix}\bOmega_{(i)}\begin{bmatrix}
\sx^{\half} & \\
&\sy^{\half}
\end{bmatrix},
\]
where 
\[
\bOmega_{(i)}=\begin{bmatrix} \bI_{p_1} & \lambda_1\bU_{(i)}\bV_{(i)}^\top+\lambda_2\bW_{(i)}\bZ_{(i)}^\top\\ \lambda_1\bV_{(i)}\bU_{(i)}^\top+\lambda_2\bZ_{(i)}\bW_{(i)}^\top& \bI_{p_2}\end{bmatrix}.
\]
Then, 
\[
D(\P_{\s_{(1)}}||\P_{\s_{(2)}})=\frac{n}{2}\Big\{\text{Tr}(\bOmega_{(2)}^{-1}\bOmega_{(1)})-(p_1+p_2)-\text{log det}(\bOmega_{(2)}^{-1}\bOmega_{(1)})\Big\}.
\]
Also notice that 
\begin{align*}
\bOmega_{(i)}=\begin{bmatrix} \bI_{p_1} & \\ & \bI_{p_2}\end{bmatrix}+\frac{\lambda_1}{2}\begin{bmatrix}
\bU_{(i)}\\
\bV_{(i)}
\end{bmatrix}\begin{bmatrix}
\bU_{(i)}^\top&\bV_{(i)}^\top
\end{bmatrix}-\frac{\lambda_1}{2}\begin{bmatrix}
\bU_{(i)}\\
-\bV_{(i)}
\end{bmatrix}\begin{bmatrix}
\bU_{(i)}^\top&-\bV_{(i)}^\top
\end{bmatrix}\\
+\frac{\lambda_2}{2}\begin{bmatrix}
\bW_{(i)}\\
\bZ_{(i)}
\end{bmatrix}\begin{bmatrix}
\bW_{(i)}^\top&\bZ_{(i)}^\top
\end{bmatrix}-\frac{\lambda_2}{2}\begin{bmatrix}
\bW_{(i)}\\
-\bZ_{(i)}
\end{bmatrix}\begin{bmatrix}
\bW_{(i)}^\top&-\bZ_{(i)}^\top
\end{bmatrix}.
\end{align*}
Therefore $\bOmega_{(1)}, \bOmega_{(2)}$ share the same set of eigenvalues: $1+\lambda_1$ with multiplicity $k$, $1-\lambda_1$ with multiplicity $k$, $1+\lambda_2$ with multiplicity $p_1-k$, $1-\lambda_2$ with multiplicity $p_1-k$ and $1$ with multiplicity $2(p_2-p_1)$. This implies $\text{log det}(\bOmega_{(2)}^{-1}\bOmega_{(1)}))=0$. On the other hand, 
by block inversion formula, we can compute 
\[
\bOmega_{(2)}^{-1}=\begin{bmatrix}
\bI_{p_1}+ \frac{\lambda_1^2}{1-\lambda_1^2}\bU_{(2)}\bU_{(2)}^\top+\frac{\lambda_2^2}{1-\lambda_2}\bW_{(2)}\bW_{(2)}^\top& -\frac{\lambda_1}{1-\lambda_1^2}\bU_{(2)}\bV_{(2)}^\top-\frac{\lambda_2}{1-\lambda_2}\bW_{(2)}\bZ_{(2)}^\top\\
 -\frac{\lambda_1}{1-\lambda_1^2}\bV_{(2)}\bU_{(2)}^\top-\frac{\lambda_2}{1-\lambda_2}\bZ_{(2)}\bW_{(2)}^\top&\bI_{p_2}+ \frac{\lambda_1^2}{1-\lambda_1^2}\bV_{(2)}\bV_{(2)}^\top+\frac{\lambda_2^2}{1-\lambda_2}\bZ_{(2)}\bZ_{(2)}^\top
\end{bmatrix}.
\]
Divide $\bOmega_{(2)}^{-1}\bOmega_{(1)}$ into blocks such that
\begin{align*}
\bOmega_{(2)}^{-1}\bOmega_{(1)}=\begin{bmatrix}
\bJ_{11} &  \bJ_{12} \\
\bJ_{21}&\bJ_{22}
\end{bmatrix}~~where ~~\bJ_{11}\in\mathbb{R}^{p_1\times p_1},~\bJ_{22}\in\mathbb{R}^{p_2\times p_2},
\end{align*}
and 
\begin{align*}
\bJ_{11}&=\frac{\lambda_1^2}{1-\lambda_1^2}(\bU_{(2)}\bU_{(2)}^\top-\bU_{(2)}\bV_{(2)}^\top\bV_{(1)}\bU_{(1)}^\top)+\frac{\lambda_2^2}{1-\lambda_2^2}(\bW_{(2)}\bW_{(2)}-\bW_{(2)}\bZ_{(2)}^\top\bZ_{(1)}\bW_{(1)}^\top)\\
&~~~-\frac{\lambda_1\lambda_2}{1-\lambda_1^2}(\bU_{(2)}\bV_{(2)}^\top\bZ_{(1)}\bW_{(1)}^\top)-\frac{\lambda_1\lambda_2}{1-\lambda_2^2}(\bW_{(2)}\bZ_{(2)}^\top\bV_{(1)}\bU_{(1)}^\top)\\
\bJ_{22}&=\frac{\lambda_1^2}{1-\lambda_1^2}(\bV_{(2)}\bV_{(2)}^\top-\bV_{(2)}\bU_{(2)}^\top\bU_{(1)}\bV_{(1)}^\top)+\frac{\lambda_2^2}{1-\lambda_2^2}(\bZ_{(2)}\bZ_{(2)}-\bZ_{(2)}\bW_{(2)}^\top\bW_{(1)}\bZ_{(1)}^\top)\\
&~~~-\frac{\lambda_1\lambda_2}{1-\lambda_1^2}(\bV_{(2)}\bU_{(2)}^\top\bW_{(1)}\bZ_{(1)}^\top)-\frac{\lambda_1\lambda_2}{1-\lambda_2^2}(\bZ_{(2)}\bW_{(2)}^\top\bU_{(1)}\bV_{(1)}^\top).
\end{align*}
We spell out the algebra for $tr(\bJ_{11})$, and $tr(\bJ_{22})$ can be computed in exactly the same fashion.
\begin{align*}
tr(\bU_{(2)}\bU_{(2)}^\top-\bU_{(2)}\bV_{(2)}^\top\bV_{(1)}\bU_{(1)}^\top)&=\frac{1}{2}tr(\bU_{(2)}\bV_{(2)}^\top\bV_{(2)}\bU_{(2)}^\top+\bU_{(1)}\bV_{(1)}^\top\bV_{(1)}\bU_{(1)}^\top-2\bU_{(2)}\bV_{(2)}^\top\bV_{(1)}\bU_{(1)}^\top)\\
&=\frac{1}{2}\|\bU_{(1)}\bV_{(1)}^\top-\bU_{(2)}\bV_{(2)}\|_F^2.
\end{align*}
Similarly, 
\begin{align*}
tr(\bW_{(2)}\bW_{(2)}-\bW_{(2)}\bZ_{(2)}^\top\bZ_{(1)}\bW_{(1)}^\top)=\frac{1}{2}\|\bW_{(1)}\bZ_{(1)}^\top-\bW_{(2)}\bZ_{(2)}\|_F^2.
\end{align*}
By the assumption \eqref{eq::kl-construction}, i.e., $\bU_{(1)}\bV_{(1)}^\top+\bW_{(1)}\bZ_{(1)}^\top=\bU_{(2)}\bV_{(2)}^\top+\bW_{(2)}\bZ_{(2)}^\top$, we have
\begin{align*}
tr(\bW_{(2)}\bW_{(2)}-\bW_{(2)}\bZ_{(2)}^\top\bZ_{(1)}\bW_{(1)}^\top)=\frac{1}{2}\|\bU_{(1)}\bV_{(1)}^\top-\bU_{(2)}\bV_{(2)}\|_F^2.
\end{align*}
Further, 
\begin{align*}
tr(\bU_{(2)}\bV_{(2)}^\top\bZ_{(1)}\bW_{(1)}^\top)&=tr\left(\bU_{(2)}\bV_{(2)}^\top(\bU_{(2)}\bV_{(2)}^\top+\bW_{(2)}\bZ_{(2)}^\top-\bU_{(1)}\bV_{(1)}^\top)^\top\right)\\
&=tr\left(\bU_{(2)}\bV_{(2)}^\top(\bU_{(2)}\bV_{(2)}^\top-\bU_{(1)}\bV_{(1)}^\top)^\top\right)\\
&=\frac{1}{2}\|\bU_{(1)}\bV_{(1)}^\top-\bU_{(2)}\bV_{(2)}\|_F^2,
\end{align*}
and by the same argument, 
\begin{align*}
tr(\bW_{(2)}\bZ_{(2)}^\top\bV_{(1)}\bU_{(1)}^\top)=\frac{1}{2}\|\bU_{(1)}\bV_{(1)}^\top-\bU_{(2)}\bV_{(2)}\|_F^2.
\end{align*}
Sum these equations, 
\begin{align*}
\text{tr}(\bJ_{11})&=\frac{1}{2}\left\{ \frac{\lambda_1^2}{1-\lambda_1^2}+\frac{\lambda_2^2}{1-\lambda_2^2}-\frac{\lambda_1\lambda_2}{1-\lambda_1^2}-\frac{\lambda_1\lambda_2}{1-\lambda_2^2}
\right\}\|\bU_{(1)}\bV_{(1)}^\top-\bU_{(2)}\bV_{(2)}\|_F^2\\
&=\frac{\Delta^2(1+\lambda_1\lambda_2)}{2(1-\lambda_1^2)(1-\lambda_2^2)}\|\bU_{(1)}\bV_{(1)}^\top-\bU_{(2)}\bV_{(2)}\|_F^2.
\end{align*}
Repeat the argument for $\bJ_{22}$, one can show that 
\begin{align*}
\text{tr}(\bJ_{22})=\text{tr}(\bJ_{11})=\frac{\Delta^2(1+\lambda_1\lambda_2)}{2(1-\lambda_1^2)(1-\lambda_2^2)}\|\bU_{(1)}\bV_{(1)}^\top-\bU_{(2)}\bV_{(2)}\|_F^2.
\end{align*}
Therefore, 
\begin{align*}
D(\P_{\s_{(1)}}||\P_{\s_{(2)}})&=\frac{n}{2} tr(\bOmega_{(2)}^{-1}\bOmega_{(1)})=\frac{n}{2}\left(tr(\bJ_{11})+tr(\bJ_{22})\right)\\
&=\frac{n\Delta^2(1+\lambda_1\lambda_2)}{2(1-\lambda_1^2)(1-\lambda_2^2)}\|\bU_{(1)}\bV_{(1)}^\top-\bU_{(2)}\bV_{(2)}\|_F^2.
\end{align*}

\bibliographystyle{chicago}
\bibliography{cca}

\begin{thebibliography}{}

\bibitem[\protect\citeauthoryear{Anderson}{Anderson}{1999}]{anderson1999asymptotic}
Anderson, T.~W. (1999).
\newblock Asymptotic theory for canonical correlation analysis.
\newblock {\em Journal of Multivariate Analysis\/}~{\em 70\/}(1), 1--29.

\bibitem[\protect\citeauthoryear{Arora and Livescu}{Arora and
  Livescu}{2013}]{arora2013multi}
Arora, R. and K.~Livescu (2013).
\newblock Multi-view cca-based acoustic features for phonetic recognition
  across speakers and domains.
\newblock In {\em Acoustics, Speech and Signal Processing (ICASSP), 2013 IEEE
  International Conference on}, pp.\  7135--7139. IEEE.

\bibitem[\protect\citeauthoryear{Cai, Ma, and Wu}{Cai
  et~al.}{2015}]{cai2015optimal}
Cai, T., Z.~Ma, and Y.~Wu (2015).
\newblock Optimal estimation and rank detection for sparse spiked covariance
  matrices.
\newblock {\em Probability theory and related fields\/}~{\em 161\/}(3-4),
  781--815.

\bibitem[\protect\citeauthoryear{Cai, Ma, Wu, et~al.}{Cai
  et~al.}{2013}]{cai2013sparse}
Cai, T.~T., Z.~Ma, Y.~Wu, et~al. (2013).
\newblock Sparse pca: Optimal rates and adaptive estimation.
\newblock {\em The Annals of Statistics\/}~{\em 41\/}(6), 3074--3110.

\bibitem[\protect\citeauthoryear{Cai and Zhang}{Cai and
  Zhang}{2017}]{cai2016rate}
Cai, T.~T. and A.~Zhang (2017).
\newblock Rate-optimal perturbation bounds for singular subspaces with
  applications to high-dimensional statistics.
\newblock {\em The Annals of Statistics, to appear\/}.

\bibitem[\protect\citeauthoryear{Chaudhuri, Kakade, Livescu, and
  Sridharan}{Chaudhuri et~al.}{2009}]{chaudhuri2009multi}
Chaudhuri, K., S.~M. Kakade, K.~Livescu, and K.~Sridharan (2009).
\newblock Multi-view clustering via canonical correlation analysis.
\newblock In {\em Proceedings of the 26th annual international conference on
  machine learning}, pp.\  129--136. ACM.

\bibitem[\protect\citeauthoryear{Chen, Liu, and Carbonell}{Chen
  et~al.}{2012}]{chen2012structured}
Chen, X., H.~Liu, and J.~G. Carbonell (2012).
\newblock Structured sparse canonical correlation analysis.
\newblock In {\em International Conference on Artificial Intelligence and
  Statistics}, pp.\  199--207.

\bibitem[\protect\citeauthoryear{Dhillon, Foster, and Ungar}{Dhillon
  et~al.}{2011}]{dhillon11}
Dhillon, P.~S., D.~Foster, and L.~Ungar (2011).
\newblock Multi-view learning of word embeddings via cca.
\newblock In {\em Advances in Neural Information Processing Systems (NIPS)},
  Volume~24.

\bibitem[\protect\citeauthoryear{Faruqui and Dyer}{Faruqui and
  Dyer}{2014}]{faruqui2014improving}
Faruqui, M. and C.~Dyer (2014).
\newblock Improving vector space word representations using multilingual
  correlation.
\newblock Association for Computational Linguistics.

\bibitem[\protect\citeauthoryear{Foster, Johnson, Kakade, and Zhang}{Foster
  et~al.}{2008}]{foster08}
Foster, D.~P., R.~Johnson, S.~M. Kakade, and T.~Zhang (2008).
\newblock Multi-view dimensionality reduction via canonical correlation
  analysis.
\newblock Technical report.

\bibitem[\protect\citeauthoryear{Friman, Borga, Lundberg, and Knutsson}{Friman
  et~al.}{2003}]{friman2003adaptive}
Friman, O., M.~Borga, P.~Lundberg, and H.~Knutsson (2003).
\newblock Adaptive analysis of fmri data.
\newblock {\em NeuroImage\/}~{\em 19\/}(3), 837--845.

\bibitem[\protect\citeauthoryear{Fukumizu, Bach, and Jordan}{Fukumizu
  et~al.}{2009}]{fukumizu2009kernel}
Fukumizu, K., F.~R. Bach, and M.~I. Jordan (2009).
\newblock Kernel dimension reduction in regression.
\newblock {\em The Annals of Statistics\/}, 1871--1905.

\bibitem[\protect\citeauthoryear{Gao, Ma, Ren, Zhou, et~al.}{Gao
  et~al.}{2015}]{gao2015minimax}
Gao, C., Z.~Ma, Z.~Ren, H.~H. Zhou, et~al. (2015).
\newblock Minimax estimation in sparse canonical correlation analysis.
\newblock {\em The Annals of Statistics\/}~{\em 43\/}(5), 2168--2197.

\bibitem[\protect\citeauthoryear{Gao, Ma, and Zhou}{Gao
  et~al.}{2017}]{gao2014sparse}
Gao, C., Z.~Ma, and H.~H. Zhou (2017).
\newblock Sparse cca: Adaptive estimation and computational barriers.
\newblock {\em The Annals of Statistics, to appear\/}.

\bibitem[\protect\citeauthoryear{Gong, Ke, Isard, and Lazebnik}{Gong
  et~al.}{2014}]{gong2014multi}
Gong, Y., Q.~Ke, M.~Isard, and S.~Lazebnik (2014).
\newblock A multi-view embedding space for modeling internet images, tags, and
  their semantics.
\newblock {\em International journal of computer vision\/}~{\em 106\/}(2),
  210--233.

\bibitem[\protect\citeauthoryear{Hom and Johnson}{Hom and
  Johnson}{1991}]{hom1991topics}
Hom, R.~A. and C.~R. Johnson (1991).
\newblock Topics in matrix analysis.
\newblock {\em Cambridge UP, New York\/}.

\bibitem[\protect\citeauthoryear{Hotelling}{Hotelling}{1936}]{hotelling36}
Hotelling, H. (1936).
\newblock Relations between two sets of variables.
\newblock {\em Biometrika\/}~{\em 28}, 312--377.

\bibitem[\protect\citeauthoryear{Kakade and Foster}{Kakade and
  Foster}{2007}]{kakade07}
Kakade, S.~M. and D.~P. Foster (2007).
\newblock Multi-view regression via canonical correlation analysis.
\newblock In {\em In Proc. of Conference on Learning Theory}.

\bibitem[\protect\citeauthoryear{Kim, Wong, and Cipolla}{Kim
  et~al.}{2007}]{kim2007tensor}
Kim, T.-K., S.-F. Wong, and R.~Cipolla (2007).
\newblock Tensor canonical correlation analysis for action classification.
\newblock In {\em Computer Vision and Pattern Recognition, 2007. CVPR'07. IEEE
  Conference on}, pp.\  1--8. IEEE.

\bibitem[\protect\citeauthoryear{Mathias}{Mathias}{1993}]{mathias1993hadamard}
Mathias, R. (1993).
\newblock The hadamard operator norm of a circulant and applications.
\newblock {\em SIAM journal on matrix analysis and applications\/}~{\em
  14\/}(4), 1152--1167.

\bibitem[\protect\citeauthoryear{Rasiwasia, Costa~Pereira, Coviello, Doyle,
  Lanckriet, Levy, and Vasconcelos}{Rasiwasia et~al.}{2010}]{rasiwasia2010new}
Rasiwasia, N., J.~Costa~Pereira, E.~Coviello, G.~Doyle, G.~R. Lanckriet,
  R.~Levy, and N.~Vasconcelos (2010).
\newblock A new approach to cross-modal multimedia retrieval.
\newblock In {\em Proceedings of the 18th ACM international conference on
  Multimedia}, pp.\  251--260. ACM.

\bibitem[\protect\citeauthoryear{Rudelson and Vershynin}{Rudelson and
  Vershynin}{2013}]{hw2013}
Rudelson, M. and R.~Vershynin (2013).
\newblock Hanson-wright inequality and sub-gaussian concentration.
\newblock {\em Electron. Commun. Probab.\/}~{\em 18}, no. 82, 1--9.

\bibitem[\protect\citeauthoryear{Sridharan and Kakade}{Sridharan and
  Kakade}{2008}]{SridharanK08}
Sridharan, K. and S.~M. Kakade (2008).
\newblock An information theoretic framework for multi-view learning.
\newblock In R.~A. Servedio and T.~Zhang (Eds.), {\em COLT}, pp.\  403--414.
  Omnipress.

\bibitem[\protect\citeauthoryear{Szarek}{Szarek}{1982}]{szarek1982nets}
Szarek, S.~J. (1982).
\newblock Nets of grassmann manifold and orthogonal group.
\newblock In {\em Proceedings of research workshop on Banach space theory (Iowa
  City, Iowa, 1981)}, Volume 169, pp.\  185.

\bibitem[\protect\citeauthoryear{Vershynin}{Vershynin}{2010}]{vershynin2010introduction}
Vershynin, R. (2010).
\newblock Introduction to the non-asymptotic analysis of random matrices.
\newblock {\em arXiv preprint arXiv:1011.3027\/}.

\bibitem[\protect\citeauthoryear{Vu, Lei, et~al.}{Vu
  et~al.}{2013}]{vu2013minimax}
Vu, V.~Q., J.~Lei, et~al. (2013).
\newblock Minimax sparse principal subspace estimation in high dimensions.
\newblock {\em The Annals of Statistics\/}~{\em 41\/}(6), 2905--2947.

\bibitem[\protect\citeauthoryear{Wang, Arora, Livescu, and Bilmes}{Wang
  et~al.}{2015}]{wang2015deep}
Wang, W., R.~Arora, K.~Livescu, and J.~Bilmes (2015).
\newblock On deep multi-view representation learning.
\newblock In {\em Proceedings of the 32nd International Conference on Machine
  Learning (ICML-15)}, pp.\  1083--1092.

\bibitem[\protect\citeauthoryear{Wedin}{Wedin}{1972}]{wedin1972perturbation}
Wedin, P.-{\AA}. (1972).
\newblock Perturbation bounds in connection with singular value decomposition.
\newblock {\em BIT Numerical Mathematics\/}~{\em 12\/}(1), 99--111.

\bibitem[\protect\citeauthoryear{Wedin}{Wedin}{1983}]{wedin1983angles}
Wedin, P.~{\AA}. (1983).
\newblock On angles between subspaces of a finite dimensional inner product
  space.
\newblock In {\em Matrix Pencils}, pp.\  263--285. Springer.

\bibitem[\protect\citeauthoryear{Witten, Tibshirani, and Hastie}{Witten
  et~al.}{2009}]{witten2009penalized}
Witten, D.~M., R.~Tibshirani, and T.~Hastie (2009).
\newblock A penalized matrix decomposition, with applications to sparse
  principal components and canonical correlation analysis.
\newblock {\em Biostatistics\/}, kxp008.

\bibitem[\protect\citeauthoryear{Yu}{Yu}{1997}]{yu1997assouad}
Yu, B. (1997).
\newblock Assouad, fano, and le cam.
\newblock In {\em Festschrift for Lucien Le Cam}, pp.\  423--435. Springer.

\end{thebibliography}
\end{document}